\documentclass[review,a4page]{elsarticle}
\usepackage[utf8]{inputenc}
\usepackage{graphicx}
\usepackage{srcltx}
\usepackage{eurosym}
\usepackage{mathtools}
\allowdisplaybreaks
\usepackage{enumerate}
\usepackage{amsmath}
\usepackage{amsfonts}
\usepackage{amssymb}
\usepackage{amsthm}
\usepackage{graphicx}
\usepackage{mathrsfs}
\usepackage{xcolor}
\usepackage{exscale}
\usepackage{latexsym}
\usepackage{esvect}
\usepackage[T1]{fontenc}

\usepackage{calc}
\usepackage[titletoc,toc,page]{appendix}
\usepackage[colorlinks,plainpages=true,pdfpagelabels,hypertexnames=true,colorlinks=true,pdfstartview=FitV,linkcolor=blue,citecolor=red,urlcolor=black]{hyperref}
\PassOptionsToPackage{unicode}{hyperref}
\PassOptionsToPackage{naturalnames}{hyperref}
\AtBeginDocument{%
\hypersetup{citecolor=red}}
\usepackage{enumerate}
\usepackage[shortlabels]{enumitem}
\usepackage{bookmark}
\usepackage{wasysym}
\usepackage{esint}
\usepackage{setspace}
\usepackage[ddmmyyyy]{datetime}
\numberwithin{equation}{section}
\everymath{\displaystyle}
\usepackage[margin=2.2cm]{geometry}

\usepackage[capitalize,nameinlink]{cleveref}
\crefname{section}{Section}{Sections}
\crefname{subsection}{Subsection}{Subsections}
\crefname{condition}{Condition}{Conditions}
\crefname{hypothesis}{Hypothesis}{Hypothesis}
\crefname{assumption}{Assumption}{Assumptions}
\crefname{lemma}{Lemma}{Lemmas}
\crefname{claim}{Claim}{Claims}
\crefname{remark}{Remark}{Remarks}

\crefformat{equation}{\textup{#2(#1)#3}}
\crefrangeformat{equation}{\textup{#3(#1)#4--#5(#2)#6}}
\crefmultiformat{equation}{\textup{#2(#1)#3}}{ and \textup{#2(#1)#3}}
{, \textup{#2(#1)#3}}{, and \textup{#2(#1)#3}}
\crefrangemultiformat{equation}{\textup{#3(#1)#4--#5(#2)#6}}%
{ and \textup{#3(#1)#4--#5(#2)#6}}{, \textup{#3(#1)#4--#5(#2)#6}}%
{, and \textup{#3(#1)#4--#5(#2)#6}}

\Crefformat{equation}{#2Equation~\textup{(#1)}#3}
\Crefrangeformat{equation}{Equations~\textup{#3(#1)#4--#5(#2)#6}}
\Crefmultiformat{equation}{Equations~\textup{#2(#1)#3}}{ and \textup{#2(#1)#3}}
{, \textup{#2(#1)#3}}{, and \textup{#2(#1)#3}}
\Crefrangemultiformat{equation}{Equations~\textup{#3(#1)#4--#5(#2)#6}}%
{ and \textup{#3(#1)#4--#5(#2)#6}}{, \textup{#3(#1)#4--#5(#2)#6}}%
{, and \textup{#3(#1)#4--#5(#2)#6}}

\crefdefaultlabelformat{#2\textup{#1}#3}
\newtheorem{theorem}{Theorem}[section]
\newtheorem{lemma}[theorem]{Lemma}

\newtheorem{definition}[theorem]{Definition}

\numberwithin{equation}{section}
\def\Yint#1{\mathchoice
{\YYint\displaystyle\textstyle{#1}}%
{\YYint\textstyle\scriptstyle{#1}}%
{\YYint\scriptstyle\scriptscriptstyle{#1}}%
{\YYint\scriptscriptstyle\scriptscriptstyle{#1}}%
\!\iint}
\def\YYint#1#2#3{{\setbox0=\hbox{$#1{#2#3}{\iint}$}
\vcenter{\hbox{$#2#3$}}\kern-.50\wd0}}
\def\longdash{-\mkern-9.5mu-} 
\def\tiltlongdash{\rotatebox[origin=c]{18}{$\longdash$}}
\def\fiint{\Yint\tiltlongdash}

\def\XXint#1#2#3{{\setbox0=\hbox{$#1{#2#3}{\int}$}
\vcenter{\hbox{$#2#3$}}\kern-.50\wd0}}

\makeatletter
\def\namedlabel#1#2{\begingroup
\def\@currentlabel{#2}%
\label{#1}\endgroup
}
\makeatother
\makeatletter
\newcommand{\rmh}[1]{\mathpalette{\raisem@th{#1}}}
\newcommand{\raisem@th}[3]{\hspace*{-1pt}\raisebox{#1}{$#2#3$}}
\makeatother

\newcommand{\descref}[2]{\hyperref[#1]{\textcolor{black}{(}\textcolor{blue}{\bf #2}\textcolor{black}{)}}}

\newcommand{\dref}[2]{\hyperref[#1]{\textcolor{black}{(}\textcolor{blue}{\bf #2}\textcolor{black}{)}}}

\makeatletter
\g@addto@macro\normalsize{%
\setlength\abovedisplayskip{3pt}
\setlength\belowdisplayskip{3pt}
\setlength\abovedisplayshortskip{1pt}
\setlength\belowdisplayshortskip{3pt}
}
\makeatletter
\def\ps@pprintTitle{%
\let\@oddhead\@empty
\let\@evenhead\@empty
\def\@oddfoot{}%
\let\@evenfoot\@oddfoot}
\makeatother
\newcounter{whitney}
\refstepcounter{whitney}

\newcounter{ineqcounter}
\refstepcounter{ineqcounter}

\begin{document}
\begin{frontmatter}
\title{Gradient H\"older regularity in mixed local and nonlocal linear parabolic problem}

\author{Stuti Das}
\ead{stutid21@iitk.ac.in and stutidas1999@gmail.com}
\address{Indian Institute of Technology Kanpur, Uttar Pradesh, 208016, India}

\newcommand*{\avint}{\mathop{\, \rlap{--}\!\!\int}\nolimits}

\begin{abstract}
We prove the local H\"older regularity of weak solutions to the mixed local nonlocal parabolic equation of the form
\begin{equation*}
u_t-\Delta u+\text{P.V.}\int_{\mathbb{R}^{n}} {\frac{u(x,t)-u(y,t)}{{\left|x-y\right|}^{n+2s}}}dy=0, 
\end{equation*}
where $0<s<1$; for every exponent $\alpha_0\in(0,1)$. Here, $\Delta$ is the usual Laplace operator. Next, we show that the gradients of weak solutions are also $\alpha$-H\"older continuous for some $\alpha\in (0,1)$. Our approach is purely analytic and it is based on perturbation techniques.
\end{abstract}
\end{frontmatter}
\begin{singlespace}
\tableofcontents
\end{singlespace}
\section{Introduction}
Mixed local nonlocal problems are emerging areas of many recent engineering works. More specifically, these kinds of problems combine two operators with different orders of differentiation, the simplest model of this case being $-\Delta+(-\Delta)^s $, for $s\in(0,1)$. In this article, our main concern is the parabolic case of such operators given by $\frac{\partial}{\partial t}- \Delta+(-\Delta)^s,$ for $s\in(0,1)$. Due to the presence of a leading local operator together with a nonlocal lower order fractional term, it is expected that the same kind of regularity results as that of the local case should also hold here. We aim to show these regularity results using a perturbation technique, where we take the solution of the heat equation as the reference one and will use their regularity. Before going into the details, we give a brief history.\\
Regularity and higher regularity results for mixed local nonlocal operators have been obtained in the past few years, mainly using iteration techniques. Using the De Giorgi-Nash-Moser theory, Garain and Kinnunen \cite{garain2022regularity} have obtained the local boundedness, local H\"older regularity, Harnack inequality results for the weak solutions of quasilinear elliptic equations of the type
\begin{equation*}
- \Delta_p u+(-\Delta_p)^s u=0 \text{ in } \Omega, \quad 1<p<\infty,
\end{equation*}
where $\Omega$ is a bounded domain in $\mathbb{R}^n$.
The higher Hölder regularity for the solution to the above equation for $2\leq p<\infty$ is obtained in \cite{garain2023higher}. The article is based on the iteration ideas of Moser. Moreover, in \cite{garain2023higher}, the authors have shown that the solution is also locally Hölder continuous for all exponents $\alpha\in(0,1)$.\\
In \cite{su2022regularity}, Xifeng Su et al. have obtained the gradient H\"older regularity of the weak solution to the mixed local nonlocal semilinear elliptic problem of form
\begin{equation*}
\begin{array}{rcl}
    - \Delta u+(-\Delta)^s u=g(x,u) &\text{ in } &\Omega,\\
    u=0& \text{ in }&\mathbb{R}^n\backslash\Omega,\end{array}
\end{equation*}
where $0<s\leq 1/2
$ and for some $c>0$, $q\in[2,2^*_s]$, the function $g$ satisfies $|g(x,t)|\leq c(1+|t|^{q-1}) \text{ for a.e. } x\in\Omega, t\in\mathbb{R}$. Here $2^*_s$ is the critical fractional Sobolev exponent given by $2^*_s=\frac{2n}{n-2s}$. Under these conditions, they have shown that the solution is in $C^{1,\alpha}(\overline{\Omega})$ for all $\alpha\in (0,1)$. The $L^\infty$-boundedness of weak solutions (for all $0<s<1$) was also obtained.\\
Whereas in \cite{MG}, the authors combined two different orders of differentiation and have shown that weak solutions to the equation of the type 
\begin{equation*}
- \Delta_p u+(-\Delta_q)^s u=f,
\end{equation*}
where $p,q>1$; are locally Hölder continuous for all $\alpha\in(0,1)$, and their gradients are Hölder continuous for some exponent $\beta\in(0,1)$, by using perturbation techniques. Along with this, they have obtained some results on the boundary regularity of weak solutions. We refer [\citealp{MG}, Section 1.1] for details on the assumptions for the problem.
\\Coming to our problem, here we deal with the regularity theory for the equation
\begin{equation}{\label{prob}}
\partial_t u-\Delta u+\mathcal{L} u=0   \text{ in } \Omega_T,
\end{equation}
where $\Omega_T=\Omega\times(0,T)$ for $T>0$, and $\mathcal{L}$ is the operator formally defined by
\begin{equation*}
\mathcal{L} u=c_{n,s}\text { P.V. } \int_{\mathbb{R}^n}\frac{u(x, t)-u(y, t)}{|x-y|^{n+2s}} d y,\qquad x \in \mathbb{R}^n , 
\end{equation*}
where $c_{n,s}$ is a suitable normalization constant, and P.V. denotes the Cauchy Principal Value.\\
Regarding our problem, we list the works done in the areas of regularity and boundedness. The weak Harnack inequality for solutions to the equation \cref{prob} have been obtained by Garain and Kinnunen in \cite{garain2023weak}. Whereas for the quasilinear case, using De Giorgi-Nash-Moser techniques, Fang et al. \cite{par1}, have obtained that weak solutions to the equation \begin{equation*}
    \partial_t u-\Delta_p u+(-\Delta_p)^s u=0    \text{ in } \Omega_T,
\end{equation*}
are locally bounded for all $1<p<\infty$ under suitable assumptions, and are locally Hölder continuous for $p>2$. Garain and Kinnunen have obtained local boundedness results and continuity results for a more general quasilinear operator \cite{ParabolicRegularity}. In \cite{par2}, Shang et al. unified the Hölder continuity results and extended it for the case $p<2$. The authors again obtained the Harnack inequality and, as a corollary, the local Hölder continuity for $p>2$ case in \cite{par3}.\smallskip\\
\textbf{{Organization of the article:}}\smallskip\\
In this article, we will follow the perturbation ideas developed by Giuseppe Mingione and Cristiana De Filippis \cite{MG}. For this, we use the regularity properties of the heat equations, which are taken from \cite{GaryM}. In the next section, we fix some notations and describe the primary function spaces we need. Moreover, the definition of weak solutions for our context, as well as some auxiliary results, are provided. 
We use the Steklov averages to define weak solutions and will eventually employ their convergence properties to estimate our results. In section 3, we prove the fundamental Caccioppoli inequality for our solution and an energy estimate in order to transfer the regularity of the reference solution (of the heat equation) to our case. We will end up getting some nonlocal integral that we also estimate there. Then, the following two sections will contain the proofs of our main results \cref{hld1,hld2}.
\section{Preliminaries}
\subsection{Notations} 
$\bullet$ $n$ denotes the dimension of the space and $z=(x, t)$ denotes a point in $\mathbb{R}^n \times(0, T)$, where $(0,T)\subset \mathbb{R}.$\\
$\bullet$ $\Omega$ denotes an open bounded domain in $\mathbb{R}^n$ with boundary $\partial \Omega$ and for $0<T <\infty$, $\Omega_T:=\Omega \times(0, T)$.\\
$\bullet$ $|\mathrm{S}|$ denotes the Lebesgue measure of a measurable subset $\mathrm{S}\subset \mathbb{R}^n$.\\
$\bullet$  The balls and parabolic cylinders are denoted by the following notations:
\begin{equation*}
\begin{array}{c}
B_\varrho\equiv B_\varrho(x_0)=\left\{x \in \mathbb{R}^n:|x-x_0|<\varrho\right\}, 
I_\varrho\equiv I_\varrho(t_0)=\left\{t \in \mathbb{R}: t_0-\varrho^2<t<t_0\right\}\\\text{ and }Q_\varrho\equiv Q_{\varrho}(z_0)=B_\varrho(x_0) \times I_\varrho(t_0) .
\end{array}
\end{equation*}
$\bullet$ $\int$ denotes the integration with respect to either space or time only, and $\iint$ denotes the integration on the product space $\Omega \times \Omega$ or $\mathbb{R}^n \times \mathbb{R}^n$.\\
$\bullet$ Also, $\iiint$ denotes the integral over some subset of $\mathbb{R}^n \times \mathbb{R}^n \times(0, T)$ and $\fint$ denotes the average integral with respect to either the space variable or the time variable. \\ 
$\bullet$ We use the notation $a \lesssim b$ for $a \leq C b$, where $C$ is a constant which depends on the dimension $n$ or $s$ or both. \\
$\bullet$ The constants $C$ or $c$, which are used throughout the article, can change their value from line to line or even in the same line but will keep dependence on certain data.\\$\bullet$ The symbol $A\Subset B$ means the set $A$ is compactly contained in $B$.
\subsection{Function spaces}
\textbf{Fractional Sobolev Spaces:} Let $0<s<1$ and $\Omega$ be a open connected subset of $\mathbb{R}^n$. The fractional Sobolev space $W^{s, 2}(\Omega)$ is defined by
\begin{equation*}
    W^{s, 2}(\Omega)=\left\{u \in L^2(\Omega): \frac{|u(x)-u(y)|}{|x-y|^{\frac{n}{2}+s}} \in L^2(\Omega\times\Omega)\right\},
\end{equation*}
and endowed with the norm
\begin{equation*}{\label{fracsob}}
\|u\|_{W^{s, 2}(\Omega)}=\left(\int_{\Omega}|u(x)|^2 d x+\int_{\Omega} \int_{\Omega} \frac{|u(x)-u(y)|^2}{|x-y|^{n+2 s}}d x d y\right)^{1/2},
\end{equation*}
it becomes a reflexive Banach space. \\
The fractional Sobolev space with zero boundary values, denoted by $W_0^{s, 2}(\Omega)$, is the subspace of $W^{s, 2}(\mathbb{R}^n)$, consisting of functions vanishing outside $\Omega$, i.e. 
\begin{equation*}
W_0^{s, 2}(\Omega)=\left\{u \in W^{s, 2}(\mathbb{R}^n): u=0 \text { in } \mathbb{R}^n \backslash \Omega\right\}.
\end{equation*}
Indeed the space $W_0^{s, 2}(\Omega)$ is also reflexive. Readers can find the details regarding fractional Sobolev spaces in [\citealp{frac}, Section 2]. The classical Sobolev space is denoted by $W^{1,2}(\Omega)$. The parabolic Sobolev space $L^2(0, T ; W^{1,2}(\Omega)),$ for $ T>0$, consists of measurable functions $u$ on $\Omega \times(0, T)$ such that
\begin{equation*}
\|u\|_{L^2\left(0, T ; W^{1,2}(\Omega)\right)}=\left(\int_0^T\|u(\cdot, t)\|_{W^{1,2}(\Omega)}^2 d t\right)^{\frac{1}{2}}<\infty.
\end{equation*}
The space $L_{\textit{loc}}^2(0, T ; W_{\textit{loc}}^{1,2}(\Omega))$ is defined by requiring the conditions above for every $\Omega^{\prime} \times[t_1, t_2] \Subset \Omega \times(0, T)$. Here $\Omega^{\prime} \times[t_1, t_2] \Subset \Omega \times(0, T)$ denotes that $\overline{\Omega^{\prime}} \times\left[t_1, t_2\right]$ is a compact subset of $\Omega \times(0, T)$.
\\\\\textbf{Tail Spaces:} The regularity results require some finiteness conditions on the nonlocal tails. To this end, 
we recall a tail space as
\begin{equation*}
L_\alpha^q(\mathbb{R}^n):=\left\{v \in L_{\textit{loc}}^q(\mathbb{R}^n) : \int_{\mathbb{R}^n} \frac{|v(x)|^q}{1+|x|^{n+\alpha}} d x<+\infty\right\}, \quad q>0 \text { and } \alpha>0 .    
\end{equation*}
Then, a nonlocal tail of the supremum version is defined by
\begin{eqnarray}{\label{tailinf}}
{\operatorname{Tail}_{\infty}(v ; x_0, r, I)} {=\operatorname{Tail}_{\infty}(v ; x_0, r, t_0-T_1, t_0+T_2)} 
 {:=\underset{t \in I}{\operatorname{ess}\text{sup}}\left(r^{2s} \int_{\mathbb{R}^n \backslash B_r(x_0)} \frac{|v(x, t)|}{\left|x-x_0\right|^{n+2s}} d x\right),}
\end{eqnarray}
where $(x_0, t_0) \in \mathbb{R}^n \times(0, T)$ and the interval $I=\left[t_0-T_1, t_0+T_2\right] \subseteq(0, T)$. From these definitions, one can easily deduce that $\operatorname{Tail}_{\infty}(v ; x_0, r, I)$ is well-defined for any $v \in L^{\infty}(I, L_{2s}^{1}(\mathbb{R}^n))$. Also we note that for any $v \in L^{\infty}(I, L_{2s}^{1}(\mathbb{R}^n))$ and for any $k\in\mathbb{R}$, the quantity $\underset{t \in I}{\operatorname{ess}\text{sup}}\left(r^{2s} \int_{\mathbb{R}^n \backslash B_r(x_0)} \frac{|v(x, t)-k|}{\left|x-x_0\right|^{n+2s}} d x\right)$ is finite.
\subsection{Weak solutions}
We denote
\begin{equation*}
    \mathcal{E}(u, v, t):=\frac{1}{2} \int_{\mathbb{R}^n} \int_{\mathbb{R}^n}(u(x, t)-u(y, t))(v(x, t)-v(y, t))\times K(x, y, t)\, d x d y,
\end{equation*}
where $K(x,y,t)=\frac{1}{{{\left|x-y\right|}^{n+2s}}}$; and define our weak solution as:
\begin{definition}{\label{weaksolll1}}
     A function $u \in  L^{2}(I; W_{\textit{loc}}^{1,2}(\Omega))
\cap C^0(I ; L_{\textit{loc}}^2(\Omega)) \cap L^{\infty}(I ; L_{2s}^{1}(\mathbb{R}^n))$ is a local weak solution to \cref{prob} if for any closed interval $I:=\left[t_1, t_2\right] \subseteq(0, T)$, there holds that
\begin{equation}{\label{soln1}}
\begin{array}{c}
 \int_{\Omega} u(x, t_2) \varphi(x, t_2)\, d x-\int_{t_1}^{t_2} \int_{\Omega} u(x, t) \partial_t \varphi(x, t) \,d x d t+\int_{t_1}^{t_2} \int_{\Omega} \nabla u\cdot\nabla \varphi \, d x d t+\int_{t_1}^{t_2} \mathcal{E}(u, \varphi, t) \,d t \\
=\int_{\Omega} u(x, t_1) \varphi(x, t_1)\, d x,
\end{array}
\end{equation}
for every function $\varphi \in L^2(I ; W^{1, 2}(\Omega)) 
\cap W^{1,2}(I ; L^2(\Omega))$ with the property that $\varphi$ has spatial support compactly contained in $\Omega$. 
\end{definition}
\subsection{Steklov averages}
\begin{definition}{\label{stekav}}
     Let $v$ be a function in $L^1(\Omega_T)$. For $0<h<T$, we define the Steklov averages $v_h(\cdot, t)$ by
\begin{equation*}
\begin{array}{l}
 v_h(\cdot,t) := \left\{\begin{array}{l}
\frac{1}{h} \int_t^{t+h} v(\cdot, \tau) d \tau, \quad t\in(0,T-h], \\
0,\quad  t>T-h;
\end{array}\right.
\end{array}
\end{equation*}
and \begin{equation*}
\begin{array}{l}
 v_{\bar{h}} (\cdot,t):= \left\{\begin{array}{l}
\frac{1}{h} \int_{t-h}^t v(\cdot, \tau) d \tau, \quad t\in(h,T],\\
0, \quad t<h.
\end{array}\right.
\end{array}
\end{equation*}
\end{definition}
We state the following lemma without proof. 
\begin{lemma}{\label{steklov}} Let $v \in L^{ r}(0,T;L^q(\Omega))$. Then, as $h\to 0$, $v_h$ converges to $v$ in $L^{ r}(0,T-\varepsilon;L^q(\Omega))$ for every $\varepsilon \in(0, T)$. If $v \in C(0, T ; L^q(\Omega))$, then as $h\to 0$, $v_h(.,t)$ converges to  $v(\cdot, t)$ in $L^q(\Omega)$ for every $t \in(0, T-\varepsilon), \forall \varepsilon \in (0,T).$
A similar statement holds for $v_{\bar{h}}$. 
\end{lemma}
We refer to [\citealp{JQ}, Lemma $2.2,2.5$] for proofs. The same proof can be done for the spaces $L^2(I ; W_{\textit{loc}}^{s, 2}(\Omega))$, $L^2(I ; W_{\textit{loc}}^{1, 2}(\Omega)).$\\\\
In view of the above lemma, we give another equivalent definition of a weak solution.
\begin{definition}{\label{weaksolll2}}
 A function $u \in  L^{2}_{\textit{loc}}(0,T; W_{\textit{loc}}^{1,2}(\Omega))
\cap C^0_{\textit{loc}}(0,T ; L_{\textit{loc}}^2(\Omega)) \cap L^{\infty}_{\textit{loc}}(0,T ; L_{2s}^{1}(\mathbb{R}^n))$ is a weak solution of \cref{prob}
if $u$ satisfies 
\begin{equation}{\label{steksol}}
\int_{\mathcal{K} \times\{t\}}\left\{u_{h, t} \varphi+(\nabla u)_h \cdot \nabla \varphi \right\} dx +\frac{1}{2} \int_{\mathbb{R}^n} \int_{\mathbb{R}^n} \left(u(x,t)-u(y,t)\right)_h\left(\varphi(x,t)-\varphi(y,t)\right)K(x,y,t)\,dxdy=0,
\end{equation}
for all $0<t<T-h$ and for all $\varphi \in W_0^{1, 2}(\mathcal{K}) 
\cap L_{\textit{loc}}^{\infty}(\Omega)$ with $\mathcal{K}\Subset \Omega$.
\end{definition}
Now we take $I =[a,b]\subset \mathbb{R}$ to be any interval and $h>0$. Given $v \in L^2(I, L^2_{\textit{loc}}(\Omega))$ we define $v_h \in C^0(I, L^2_{\textit{loc}}(\Omega))$ by
\begin{equation*}
v_h(\cdot, t)=\frac{1}{h} \int_t^{t+h} \tilde{v}(\cdot, s) d s, \quad \text { for each } a \leq t \leq b,
\end{equation*}
where $\tilde{v}(\cdot, t) \in L^2(\mathbb{R}, L^2_{\textit{loc}}(\Omega))$ is defined by
\begin{equation*}
\tilde{v}(\cdot, t)=\left\{\begin{aligned}
&v(\cdot, t), \quad\text { when } t \in I, \\
&0, \quad \text { when } t \in \mathbb{R} \backslash I .
\end{aligned}\right.
\end{equation*}
With these definitions, we state the following important result [\citealp{JQ}, Lemma 2.4] regarding Steklov averages:
\begin{lemma}{\label{steklov2}} $v_h(\cdot, t)$ as defined above satisfies for all $E\Subset \Omega$:
\begin{equation*}
    v_h(\cdot, t) \in L^2(E), \quad \forall t \in I
\end{equation*}
with
\begin{equation*}
\begin{gathered}
\left\|v_h(\cdot, t)\right\|_{L^2(E)} \leq \frac{1}{h^{1 / r}}\|v(\cdot, t)\|_{L^2\left(I, L^2(E)\right)}; \\
v_h
\in C^0(I, L^2(E)) \cap L^{\infty}(I, L^2(E)) \text{ and }
v_h(\cdot, t): I \rightarrow L^2(E) \text { is uniformly continuous on } I .\\
v_h
\in L^2(I,L^2(E))
\text{ and } ||v_h||_{L^2(I,L^2(E))}\leq ||v||_{L^2(I,L^2(E))}.
\end{gathered}
\end{equation*}
\end{lemma}
The same result as of \cref{steklov2} can be done for the space $L^2(I ; W_{\textit{loc}}^{1, 2}(\Omega))$ and $L^2(I ; W_{\textit{loc}}^{s, 2}(\Omega))$.
\subsection{Statements of the main results} We state the main theorems of this article as the following.
\begin{theorem}{\label{hld1}}
    Assume that $u$ is a locally bounded local weak solution of \cref{prob} in $\Omega_T$, then $u$ is locally H\"older continuous in $\Omega_T$ for all exponents $\alpha_0\in(0,1).$ More precisely for every $\alpha_0\in(0,1)$, there exists $r_*\equiv r_*(n,s,\alpha_0)\in (0,1)$ and $c \equiv c (n,s,\alpha_0
)
 \geq 1 $ such that
\begin{equation*}{\label{continuitystate}}
    \left(\fint_{t_0-r^2}^{t_0} \fint_{B_r\left(x_0\right)}\left|u(x,t)-(u)_{B_r(x_0)\times(t_0-r^2,t_0)}\right|^2 d x dt\right)^{\frac{1}{2}}\leq c\left(\frac{r}{r_*}\right)^{\alpha_0}\left(\operatorname{Tail}_{\infty}(u;x_0,r_*,t_0-r_*^2,t_0)+||u||_{L^2\left(\Omega_T\right)}\right)
\end{equation*}
holds whenever $B_r\Subset \Omega$, $(t_0-r^2,t_0)\Subset (0,T)$ and $0<r\leq r^*$.
\end{theorem} 
\begin{theorem}{\label{hld2}}
Let $u$ be a locally bounded local weak solution of \cref{prob} in $\Omega_T$, then $\nabla u$ is locally H\"older continuous in $\Omega_T$. More precisely, $\exists \alpha\equiv\alpha(s,n) \in(0,1)$ and $\tilde{r}\equiv \tilde{r}(n,s)\in(0,1)$, such that for every $\Omega_0\Subset \Omega$ and $(\tau_1,\tau_2)\Subset(0,T)$, $\|\nabla u\|_{C^{0,\alpha}(\Omega_0\times(\tau_1,\tau_2))}\leq c$ holds with $c \equiv c(n,s,\alpha_0,||u||_{L^2(\Omega_T)},||u||_{\operatorname{Tail}_{\infty}(u;x_0,t_0,\tilde{r})}, \operatorname{dist}\{\Omega_0,\partial \Omega\},\tau_1,\tau_2,T
 )$.
    \end{theorem}
\subsection{Auxiliary results}
We begin with the basic Poincar\'{e} inequality, which can be found in [\citealp{LCE}, Chapter 5, Section 5.8.1].
\begin{lemma}{\label{p}}
  Let $\Omega\subset \mathbb{R}^n$ be a bounded domain with $C^1$ boundary. Then there exist a positive constant $C>0$ depending only on $n$ and $\Omega$, such that \begin{equation*} 
  \int_\Omega \left|u\right|^2 dx\leq C\int_\Omega \left|\nabla u\right|^2 dx, \qquad\forall u\in W^{1,2}_0(\Omega).
  \end{equation*}
  Specifically, if we take $\Omega=B_\varrho$, then we will get for all $u\in W^{1,2}(B_\varrho)$,
  \begin{equation*}
  \fint_{B_{\varrho}}\left|u-(u)_{B_\varrho}\right|^2dx \leq c \varrho^2 \fint_{B_{\varrho}}\left|\nabla u\right|^2dx,
  \end{equation*}
  where $c$ is a constant depending only on $n$ and $(u)_{B_\varrho}$ denotes the average of $u$ in $B_\varrho$.
 \end{lemma}
 The following version of Poincar\'{e} inequality holds for fractional Sobolev spaces.
 \begin{lemma}{\label{fracpoin}} Let $s \in(0,1), B_{\varrho} \subset \mathbb{R}^n$ be a ball. If $u \in W^{s,2}(B_{\varrho})$, then
\begin{equation*}
\left(\fint_{B_{\varrho}}\left|u-(u)_{B_\varrho}\right|^2 d x\right)^{1 / 2} \leq c \varrho^s\left(\int_{B_{\varrho}} \fint_{B_{\varrho}} \frac{|u(x)-u(y)|^2}{|x-y|^{n+2s}} d x d y\right)^{1 / 2}
\end{equation*}
holds with $c \equiv c(n, s)$.
\end{lemma}
The following result asserts that the classical Sobolev space is continuously embedded in the fractional Sobolev space; for details, one can see [\citealp{frac}, Proposition 2.2]. The idea applies an extension property of $\Omega$ so that we can extend functions from $W^{1,2}(\Omega)$ to $W^{1,2}(\mathbb{R}^n)$ and that the extension operator is bounded.
\begin{lemma}{\label{embedding}}
    Let $\Omega$ be a smooth bounded domain in $\mathbb{R}^n$ and $0<s<1$. There exists a positive constant $C=C(\Omega, n, s)$ such that
\begin{equation*}
\|u\|_{W^{s, 2}(\Omega)} \leq C\|u\|_{W^{1,2}(\Omega)},
\end{equation*}
for every $u \in W^{1,2}(\Omega)$.
\end{lemma} 
The next embedding result for the fractional Sobolev spaces with zero boundary value follows from [\citealp{frac2}, Lemma 2.1]. The fundamental difference of it compared to \cref{embedding} is that the result holds for any bounded domain (without any condition of smoothness on the boundary) since for the Sobolev spaces with zero boundary value, we always have a zero extension to the complement.
\begin{lemma}{\label{embedding2}} Let $\Omega$ be a bounded domain in $\mathbb{R}^n$ and $0<s<1$. There exists a positive constant $C=C(n, s, \Omega)$ such that
\begin{equation*}
\int_{\mathbb{R}^n} \int_{\mathbb{R}^n} \frac{|u(x)-u(y)|^2}{|x-y|^{n+2 s}} d x d y \leq C \int_{\Omega}|\nabla u|^2 d x,
\end{equation*}
for every $u \in W_0^{1,2}(\Omega)$. Here, we consider the zero extension of $u$ to the complement of $\Omega$.
\end{lemma}
One particular version of \cref{embedding2} can be found in [\citealp{MG}, Lemma 2.2].
\begin{lemma}{\label{fracemb}} Let $B_{\varrho} \subset \mathbb{R}^n$ be a ball. If $u \in W_0^{1, 2}(B_{\varrho})$, then $u \in W^{s, 2}(B_{\varrho})$ and
\begin{equation*}
\left(\int_{B_{\varrho}} \fint_{B_{\varrho}} \frac{|u(x)-u(y)|^2}{|x-y|^{n+2s}} d x d y\right)^{1 / 2} \leq c \varrho^{1-s}\left(\fint_{B_{\varrho}}|\nabla u|^2 d x\right)^{1 / 2}
\end{equation*}
holds with $c \equiv c(n,  s)$.
\end{lemma}
The following useful lemma can be found in [\citealp{MG}, Lemma 3.2].
\begin{lemma}{\label{Ming}} Let $s\in(0,1)$, 
$w \in L_{\textit{loc}}^2(\mathbb{R}^n)$ and $B_\varrho(x_0) \subset \mathbb{R}^n$ be a ball. Then there holds
\begin{equation*}
\int_{\mathbb{R}^n \backslash B_\varrho} \frac{|w(y)|}{\left|y-x_0\right|^{n+2s}} d y \leq \frac{c}{\varrho^s}\left(\int_{\mathbb{R}^n \backslash B_\varrho} \frac{|w(y)|^2}{\left|y-x_0\right|^{n+2s}}dy\right)^{1/2},
\end{equation*}
where $c \equiv c(n, s)$.
\end{lemma}
The following lemma can be found in [\citealp{DC}, Lemma 3.1].
\begin{lemma}{\label{abp}} Let $p \geq 1$. For $a, b \in \mathbb{R}$ and $\varepsilon>0$, we have that
\begin{equation*}
|a|^p \leq|b|^p+C_p \varepsilon|b|^p+\left(1+C_p \varepsilon\right) \varepsilon^{1-p}|a-b|^p,
\end{equation*}
with $C_p:=(p-1) \Gamma(\max \{1, p-2\})$. Here $\Gamma$ is the standard Gamma function.
\end{lemma}
The following version of parabolic Poincar\'{e} is from [\citealp{KA}, Lemma 2.13].
\begin{lemma}{\label{Poincar\'{e}}} Let $f \in L^{2}(0, T ; W^{1, 2}(\Omega))$ and suppose that ${B}_r \Subset \Omega$ be a compactly contained ball of radius $r>0$. Let $I \subset(0, T)$ be a time interval and $\rho(x, t) \in L^1({B}_r \times I)$ be a positive function satisfying $\|\rho\|_{L^{\infty}({B}_r \times I)} \lesssim_n \frac{\|\rho\|_{L^1({B}_r \times I)}}{\left|{B}_r \times I\right|}$ and $\mu(x) \in C_c^{\infty}({B}_r)$ such that $\int_{{B}_r} \mu(x) d x=1$ with $|\mu| \leq \frac{C(n)}{r^n}$ and $|\nabla \mu| \leq \frac{C(n)}{r^{n+1}}$. Then there holds
\begin{equation*}
\begin{array}{l}
 \qquad \fint\fint_{{B}_r \times I}\left|\frac{f-(f)_\rho}{r}\right|^{2} d z \quad \lesssim_{n} \fint\fint_{{B}_r \times I}|\nabla f|^{2} d z+\sup _{t_1, t_2 \in I}\left|\frac{(f)_\mu(t_2)-(f)_\mu(t_1)}{r}\right|^{2},\\
\text{ where } 
    (f)_\rho:=\iint_{{B}_r \times I} f(z) \frac{\rho(z)}{\|\rho\|_{L^1({B}_r \times I)}} d z \text { and }(f)_\mu(t_i):=\int_{{B}_r} f(x, t_i) \mu(x) d x \text { for } i=1,2 .
    \end{array}
\end{equation*}
\end{lemma}
\section{Fundamental Estimates}
\subsection{Caccioppoli estimate}
\begin{theorem}\label{energy_estimate}
	Let $u$ be a local solution to our problem. Let $x_0 \in \Omega, r>0, B_r \equiv B_r(x_0)$ satisfies $\bar{B}_r \subseteq \Omega$ and $0<\tau_1<\tau_2,$ satisfies $\left[\tau_1, \tau_2\right] \subseteq(0, T)$. Then for all functions $\psi \in C_0^{\infty}(B_r)$ and $\eta \in C^{\infty}(\mathbb{R})$ such that $\eta(t) \equiv 0$ if $t \leq \tau_1$, 
  there exists a constant $C>0$ only depending on $s$ and $n$ such that
	\begin{equation*}
	\begin{array}{l}
		\quad \int_{\tau_1}^{\tau_2} \int_{B_r} \int_{B_r}\left|v(x, t) \psi(x)-v(y, t) \psi(y)\right|^2 \eta^2(t) d \mu d t+\underset{\tau_1<t<\tau_2}{\operatorname{ess} \sup} \int_{B_r} v^2(x, t) \psi^2(x) \eta^2(t) d x \smallskip\\\quad+\int_{\tau_1}^{\tau_2}\int_{B_r} \left|\nabla u(x,t) \right|^2\psi ^2(x) \eta^2(t)dxdt\smallskip\\
		\leq C \int_{\tau_1}^{\tau_2} \int_{B_r} \int_{B_r} \operatorname{max} \left\{v(x, t), v(y, t)\right\}^2|\psi(x)-\psi(y)|^2 \eta^2(t) d \mu d t \smallskip \\\quad+C \int_{\tau_1}^{\tau_2}\int_{B_r}\left(\int_{\mathbb{R}^n \backslash B_r} \frac{|v(y, t)|}{|x-y|^{n+2s}} d y \right)|v(x, t)| \psi^2(x) \eta^2(t) d x d t \smallskip\\\quad+C \int_{\tau_1}^{\tau_2} \int_{B_r} v^2(x, t) \psi^2(x) \eta(t)\left|\partial_t \eta(t)\right| d x d t+C\int_{\tau_1}^{\tau_2}\int_{B_r}v^2(x,t)\left|\nabla \psi (x) \right|^2 \eta^2(t)d x d t,
	\end{array}
	\end{equation*}
	where $d \mu:= K(x,y,t)dxdy$ and $v:=
u-k$ with a level $k \in \mathbb{R}$. 
\end{theorem}
\begin{proof} Take
$
0<h<T-\tau_2
$ and $I=[\tau_1,\tau_2].$ We choose arbitrary $\varphi \in L^2(I ; W^{1, 2}(\Omega)) 
\cap W^{1,2}(I ; L^2(\Omega))$ such that $\varphi$ has spatial support compactly contained in $B_r$. So from \cref{steksol}, integrating over $(t_1,t_2]\subset(\tau_1,\tau_2]$ we obtain
\begin{equation*}
 \int_{t_1} ^{t_2}\int_{B_r} u_{h,t}\varphi d xdt+\int_{t_1}^{t_2} \int_{B_r} (\nabla u)_h\cdot\nabla \varphi dxdt+\int_{t_1}^{t_2} \mathcal{E}(u_h, \varphi, t) d t  =0.
\end{equation*}
Let us fix here $t_1=\tau_1$ and take $t_2 \in\left(\tau_1, \tau_2\right]$ which will be specified later.\\
Now as $v(x, t):=(u-k)(x, t)$, we choose $\varphi(x, t)=v_h(x, t) \psi^2(x) \eta^2(t)$ in above to get
\begin{equation}{\label{ca1}}
    \begin{array}{c}
         \underbrace{\int_{t_1}^{t_2} \int_{B_r} \partial_t u_h(x, t)(v_h \psi^2 \eta^2)(x, t) d x d t}_{I_1^h}\\+\underbrace{\frac{1}{2} \int_{t_1}^{t_2} \int_{B_r} \int_{B_r} (v(x, t)-v(y, t))_h ((v_h \psi^2 \eta^2)(x, t)-(v_h \psi^2 \eta^2)(y, t)) d \mu d t}_{I_2^h} \\
 +\underbrace{\int_{t_1}^{t_2} \int_{\mathbb{R}^n \backslash B_r} \int_{B_r} (v(x, t)-v(y, t))_h \times(v_h\psi^2 \eta^2)(x, t) d \mu d t}_{I_3^h}
+\underbrace{\int_{t_1}^{t_2} \int_{B_r} (\nabla u)_h\cdot\nabla (v_h\psi^2\eta^2) dxdt}_{I_4^h}=0.
    \end{array}
\end{equation}
Now we need to take $h \rightarrow 0$ and find the limits of $I_1^h, I_2^h, I_3^h, I_4^h$ before we proceed for our required estimate.
\begin{description}[leftmargin=0cm]
\item{\textbf{Estimate for $I_1^h$:}} Using integrating by parts, we obtain that
\begin{equation*}
\begin{array}{rcl}
I_1^h  &=&\int_{t_1}^{t_2} \int_{B_r} \partial_t(v_h(x, t)) v_h(x, t) \psi^2(x) \eta^2(t) d x d t\smallskip \\
& =&\frac{1}{2} \int_{t_1}^{t_2} \int_{B_r} \partial_t(v_h(x, t))^2 \psi^2(x) \eta^2(t) d x d t\smallskip\\
& =& \frac{1}{2} \int_{B_r}\left(v_h(x, t_2)\right)^2 \psi^2(x) \eta^2(t_2) d x  -\frac{1}{2} \int_{B_r}\left(v_h(x, t_1)\right)^2 \psi^2(x) \eta^2(t_1) d x\smallskip\\&& -\int_{t_1}^{t_2} \int_{B_r}\left(v_h(x, t)\right)^2 \psi^2(x) \eta(t) \partial_t \eta(t) d x d t .
\end{array}
\end{equation*}
Since our weak solution $u \in C^0(I; L_\textit{loc}^2(\Omega))$, $\psi$ is bounded, and $\eta$ is also bounded over $I$, therefore using \cref{steklov},
we get the following convergence
\begin{equation}{\label{I1}}
\begin{array}{c}
    I_1^h \longrightarrow  \frac{1}{2} \int_{B_r} v^2(x, t_2) \psi^2(x) \eta^2(t_2)d x-\frac{1}{2} \int_{B_r} v^2(x, t_1) \psi^2(x) \eta^2(t_1) d x \\
 -\int_{t_1}^{t_2} \int_{B_r} \eta(t) \partial_t \eta(t) \psi^2(x) v^2(x, t) d x d t =:I_1\quad\text { as } h \rightarrow 0.
\end{array}
\end{equation}
\item{\textbf{Estimate for $I_2^h$:}} Using the form of $K(x,y,t)$, we get
\begin{equation*}
I_2^h= \frac{1}{2} \int_{t_1}^{t_2} \int_{B_r} \int_{B_r} (v(x, t)-v(y, t))_h K^{\frac{1}{2}}(x,y,t) \times((v_h \psi^2\eta^2)(x, t)-(v_h \psi^2 \eta^2)(y, t)) K^{\frac{1}{2}}(x,y,t) \,dx dy d t.
\end{equation*}
Using \cref{steklov2} we see that the norm of $\left\{v_h \right\}_{h \in\left(0, T-\tau_2\right)}$ in $L^2(t_1, t_2 ; W^{s, 2}(B_r))$ is bounded by $||v||_{L^2(t_1, t_2 ; W^{s, 2}(B_r))}$, which is independent of $h$; also since $\psi$ and $\eta$ are bounded, we will get
\begin{equation*}
\begin{array}{rcl}
&\left\|\frac{(v_h \psi^2 \eta^2)(x, t)}{|x-y|^{\frac{n}{2}+s}}\right\|_{L^2\left(t_1, t_2 ; L^2\left(B_r \times B_r\right)\right)}& \leq C 
\smallskip\\
\implies&
\left\|K^{\frac{1}{2}}(x, y, t)(v_h \psi^2 \eta^2)(x, t)\right\|_{L^2\left(t_1, t_2 ; L^2\left(B_r \times B_r\right)\right)} &\leq C .
\end{array}
\end{equation*}
Now as $(v_h \psi^2 \eta^2) \rightarrow v \psi ^2\eta^2$ a.e. in $B_r \times(t_1, t_2)$ we derive that as $h \rightarrow 0$,
\begin{equation*}
    K^{\frac{1}{2}}(x, y, t)(v_h \psi^2 \eta^2)(x, t)  \rightharpoonup K^{\frac{1}{2}}(x, y, t) v(x, t) \psi^2(x) \eta^2(t) \text { in } L^2(t_1, t_2 ; L^2(B_r \times B_r)) .
\end{equation*}
By utilizing form of $K(x,y,t)$, we can check
\begin{equation*}
(v(x, t)-v(y, t))_h K^{\frac{1}{2}}(x, y, t) \in L^{2}(t_1, t_2 ; L^{2}(B_r \times B_r)) .
\end{equation*}
Thus, combining the above three, we finally get the convergence that
$
I_2^h \rightarrow I_2  \text { as } h \rightarrow 0 \text {,}
$
where 
\begin{equation}{\label{I2}}
    I_2=\frac{1}{2} \int_{t_1}^{t_2} \int_{B_r} \int_{B_r} (v(x, t)-v(y, t)) \times(v(x, t) \psi^2(x)-v(y, t) \psi^2(y)) \eta^2(t) d \mu d t .
\end{equation}
\item{\textbf{Estimate for $I_3^h$:}} For the term $I_3^h$, using similar arguments as of the case $I_2^h$, we can derive that
\begin{equation}{\label{I3}}
    I_3^h \rightarrow \int_{t_1}^{t_2} \int_{\mathbb{R}^n \backslash B_r} \int_{B_r} (v(x, t)-v(y, t)) \times(v \psi^2 \eta^2)(x, t) d \mu d t=: I_3 \quad \text { as } h \rightarrow 0 .
\end{equation}
The explicit reasoning for proving the above can be found in [\citealp{BL}, Appendix B].
\item{\textbf{Estimate for $I_4^h$:}} Since the Steklov averages have nothing to do with the variable $x$, we get the following convergence
\begin{equation}{\label{I4}}
I_4^h \rightarrow \int_{t_1}^{t_2} \int_{B_r} \nabla u\cdot\nabla (v\psi^2\eta^2) dxdt=: I_{4} \quad \text  { as } h \rightarrow 0.
\end{equation}
\end{description}
Finally, by virtue of \cref{I1,I2,I3,I4} we get from \cref{ca1} that
\begin{equation}{\label{IIII}}
I_1+I_2+I_3 + I_4=0 .    
\end{equation}
Now, in order to get the desired estimate, we start estimating $I_1, I_2, I_3$ and $I_4$.
\begin{description}[leftmargin=0cm]
\item{\textbf{The estimate for $I_1$:}} Since we assumed that $\eta(t_1)=0$, we directly have from \cref{I1} that
\begin{equation}{\label{I11}}
    I_1=\frac{1}{2} \int_{B_r} v^2(x, t_2) \psi^2(x) \eta^2(t_2) d x-\int_{t_1}^{t_2} \int_{B_r} \eta(t) \partial_t \eta(t) \psi^2(x) v^2(x, t) d x d t.
\end{equation}
\item{\textbf{The estimate for $I_2$:}} 
For $I_2$, let us set
\begin{equation*}
    \mathcal{T}(x,y)=(v(x,t)-v(y,t))(v(x,t)\psi^2(x)-v(y,t)\psi^2(y))\eta^2(t).
\end{equation*}
We first assume that $\psi(x)\geq\psi(y)$ and rewrite $\mathcal{T}(x,y)=\mathcal{T}_1(x,y)+\mathcal{T}_2(x,y)$, where 
\begin{equation*}
\begin{array}{c}
    \mathcal{T}_1(x,y)=(v(x,t)-v(y,t))^2\psi^2(x)\eta^2(t) \end{array}\end{equation*} 
    and \begin{equation*}
\begin{array}{c}
\qquad\mathcal{T}_2(x,y)=(v(x,t)-v(y,t))(\psi^2(x)-\psi^2(y))v(y,t)\eta^2(t).
    \end{array}
\end{equation*}
Now we get using the fact that $\psi(x)\geq\psi(y)$, we get
\begin{equation*}
    |\mathcal{T}_2(x,y)|\leq 2|\psi(x)||\psi(x)-\psi(y)||v(x,t)-v(y,t)||v(y,t)|\eta^2(t)
\end{equation*}
and by Young's inequality, we obtain
\begin{equation*}
\begin{array}{rcl}
     \mathcal{T}_1(x,y)&=&\mathcal{T}(x,y)-\mathcal{T}_2(x,y)\\
     &\leq &\mathcal{T}(x,y)+|\mathcal{T}_2(x,y)| \\&\leq& \mathcal{T}(x,y)+2|\psi(x)||\psi(x)-\psi(y)||v(x,t)-v(y,t)||v(y,t)|\eta^2(t)\\&\leq& \mathcal{T}(x,y)+2\varepsilon\psi^2(x)|v(x,t)-v(y,t)|^2\eta^2(t)+2C(\varepsilon)|\psi(x)-\psi(y)|^2|v(y,t)|^2\eta^2(t)
    \\& =&\mathcal{T}(x,y)+2\varepsilon\mathcal{T}_1(x,y)+2C(\varepsilon)|\psi(x)-\psi(y)|^2|v(y,t)|^2\eta^2(t).
\end{array}
\end{equation*}
Now choosing $\varepsilon=1/4$ we get 
\begin{equation*}
\begin{array}{rcl}
  |v(x,t)-v(y,t)|^2\psi^2(y)\eta^2(t)&\leq&  |v(x,t)-v(y,t)|^2\psi^2(x)\eta^2(t)\\&\leq& 2 \mathcal{T}(x,y) +C|\psi(x)-\psi(y)|^2|v(y,t)|^2\eta^2(t) ,
  \end{array}
\end{equation*}
with $C$ being a constant. Now, if $\psi(y)>\psi(x)$, we note that $\mathcal{T}(x,y)=\mathcal{T}(y,x)$ and interchanging the role of $x$ and $y$ in the above estimates, in any case we conclude that 
\begin{equation*}
    |v(x,t)\psi(x)-v(y,t)\psi(y)|^2\eta^2(t)\leq 2 \mathcal{T}(x,y) +C\operatorname{max}\{|v(x,t)|,|v(y,t)|\}^2|\psi(x)-\psi(y)|^2\eta^2(t) .
\end{equation*}
Finally, using \cref{I2}, we get
\begin{equation}{\label{I22}}
\begin{array}{rcl}
    I_2&=&\frac{1}{2} \int_{t_1}^{t_2} \int_{B_r} \int_{B_r} (v(x, t)-v(y, t)) \times(v(x, t) \psi^2(x)-v(y, t) \psi^2(y)) \eta^2(t) d \mu d t\smallskip \\
 &\geq& \frac{1}{4} \int_{t_1}^{t_2} \int_{B_r} \int_{B_r}\left|v(x, t) \psi(x)-v(y, t) \psi(y)\right|^2 \eta^2(t) d \mu d t \smallskip\\
 &&-C \int_{t_1}^{t_2} \int_{B_r} \int_{B_r}\left(\max \left\{v(x, t), v(y, t)\right\}\right)^2|\psi(x)-\psi(y)|^2 \eta^2(t) d \mu d t .  
\end{array}
\end{equation}
\item{\textbf{The estimate for $I_3$:}} Since
$
(v(x, t)-v(y, t)) v(x, t) \geq-v(y, t)v(x, t) \geq-|v(y,t)||v(x,t)|,
$
it follows from \cref{I3} that
\begin{equation}{\label{I33}}
    \begin{array}{rcl}
         I_3 & =&\int_{t_1}^{t_2} \int_{\mathbb{R}^n \backslash B_r} \int_{B_r}(v(x, t)-v(y, t)) v(x, t) \psi^2(x) \eta^2(t) d \mu d t \\&\geq&-C \int_{t_1}^{t_2} \int_{B_r}\left(\int_{\mathbb{R}^n \backslash B_r} \frac{|v(y, t)|}{|x-y|^{n+2s}} d y\right) |v(x, t)| \psi^2(x) \eta^2(t) d x d t .
    \end{array}
\end{equation}
\item{\textbf{The estimate for $I_4$:}} Using integration by parts in \cref{I4}, we get
\begin{equation*}
\begin{array}{rcl}
   I_4&=&\int_{t_1}^{t_2} \int_{B_r} \nabla v\cdot\nabla (v\psi^2\eta^2) dxdt
   = \int_{t_1}^{t_2} \int_{B_r} \left|\nabla v\right|^2\psi^2\eta^2 dxdt+2\int_{t_1}^{t_2} \int_{B_r}v \nabla v \cdot \psi \nabla{\psi} \eta^2 dx dt.
   \end{array}
\end{equation*}
Using Young's inequality with $\varepsilon$, we get by \cref{IIII},
\begin{equation}{\label{I44}}
    \begin{array}{rcl}
           I_1+I_2+I_3+ \int_{t_1}^{t_2} \int_{B_r} \left|\nabla v\right|^2\psi^2\eta^2 dxdt  &=& -2\int_{t_1}^{t_2} \int_{B_r}v \nabla v \cdot \psi \nabla{\psi} \eta^2 dx dt\\& \leq& 2\int_{t_1}^{t_2} \int_{B_r}\left| v \nabla{\psi} \right | \left|\psi \nabla v  \right| \eta^2 dx dt\smallskip\\&
    \leq& \varepsilon \int_{t_1}^{t_2} \int_{B_r} \left|\nabla v\right|^2\psi^2\eta^2 dxdt + C(\varepsilon ) \int_{t_1}^{t_2} \int_{B_r} \left|\nabla \psi \right|^2v^2\eta^2 dxdt .
    \end{array}
\end{equation}   
\end{description}
Now choosing $\varepsilon=1/2$ and using \cref{I11,I22,I33,I44}, we get from \cref{IIII},
\begin{equation*}
\begin{array}{l}
\int_{B_r} v^2(x, t_2) \psi^2(x) \eta^2(t_2) d x+\int_{\tau_1}^{t_2} \int_{B_r} \int_{B_r}\left|v(x, t)\psi(x)-v(y, t)\psi(y)\right|^2 \eta^2(t) d \mu d t\\\qquad\qquad\qquad\qquad\qquad+ \int_{\tau_1}^{t_2} \int_{B_r} \left|\nabla v\right|^2\psi^2\eta^2 dxdt\\
 \leq C \int_{\tau_1}^{\tau_2} \int_{B_r} \int_{B_r}\left(\max \left\{v(x, t), v(y, t)\right\}\right)^2|\psi(x)-\psi(y)|^2 \eta^2(t) d \mu d t\\\qquad\qquad\qquad+C \int_{\tau_1}^{\tau_2} \int_{B_r} v^2(x, t) \psi^2(x) \eta(t)\left|\partial_t \eta(t)\right| d x d t  \\
+C \int_{\tau_1}^{\tau_2} \int_{B_r}\left(\int_{\mathbb{R}^n \backslash B_r} \frac{|v(y, t)|}{|x-y|^{n+2s}} d y \right) |v(x, t)| \psi^2(x) \eta^2(t) d x d t+C\int_{\tau_1}^{\tau_2} \int_{B_r} v^2  \left|\nabla{\psi}\right|^2 \eta^2 dx dt.
\end{array}
\end{equation*}
In the above, since the right hand side is positive and independent of $t_2$, therefore we can vary $t_2\in(\tau_1,\tau_2]$ in the left (this side is also positive) and take essential supremum over $t_2\in(\tau_1,\tau_2]$ to get the desired estimate
\begin{equation}{\label{Caccioppoli}}
	\begin{array}{l}
 \int_{\tau_1}^{\tau_2} \int_{B_r} \int_{B_r}\left|v(x, t) \psi(x)-v(y, t) \psi(y)\right|^2 \eta^2(t) d \mu d t+\underset{\tau_1<t<\tau_2}{\operatorname{ess} \sup} \int_{B_r} v^2(x, t) \psi^2(x) \eta^2(t) d x \\
\qquad \qquad \qquad \qquad +\int_{\tau_1}^{\tau_2}\int_{B_r} \left|\nabla u(x,t) \right|^2\psi ^2(x) \eta^2(t)dxdt\\
 \leq C \int_{\tau_1}^{\tau_2} \int_{B_r} \int_{B_r} \max \left\{v(x, t), v(y, t)\right\}^2|\psi(x)-\psi(y)|^2 \eta^2(t) d \mu d t \\
 \hspace*{1.3cm}\quad +C \int_{\tau_1}^{\tau_2}   \int_{B_r}\left(\int_{\mathbb{R}^n \backslash B_r} \frac{|v(y, t)|}{|x-y|^{n+2s}} d y \right) |v(x, t)| \psi^2(x) \eta^2(t) d x d t \\
  \quad +C \int_{\tau_1}^{\tau_2} \int_{B_r} v^2(x, t) \psi^2(x) \eta(t)\left|\partial_t \eta(t)\right| d x d t+C\int_{\tau_1}^{\tau_2}\int_{B_r}v^2(x,t)\left|\nabla \psi (x) \right|^2 \eta^2(t) d x d t,
\end{array}
\end{equation}
	where $d \mu:= K(x,y,t)dxdy$ and $v:=u-k$ with a level $k \in \mathbb{R}$.
\end{proof}
\subsection{Improving the Caccioppoli estimate}In \cref{Caccioppoli}, taking average integral we get
\begin{equation}
	\begin{array}{l}
 \fint_{\tau_1}^{\tau_2} \int_{B_r} \fint_{B_r}\left|v(x, t) \psi(x)-v(y, t) \psi(y)\right|^2 \eta^2(t) d \mu d t+\frac{1}{(\tau_2-\tau_1)}\underset{\tau_1<t<\tau_2}{\operatorname{ess} \sup} \fint_{B_r} v^2(x, t) \psi^2(x) \eta^2(t) d x \\
\qquad \qquad \qquad \qquad +\fint_{\tau_1}^{\tau_2}\fint_{B_r} \left|\nabla u(x,t) \right|^2\psi ^2(x) \eta^2(t)dxdt\\
 \leq C \fint_{\tau_1}^{\tau_2} \int_{B_r} \fint_{B_r} \max \left\{v(x, t), v(y, t)\right\}^2|\psi(x)-\psi(y)|^2 \eta^2(t) d \mu d t \\
 \hspace*{1.3cm}\quad +C \fint_{\tau_1}^{\tau_2}  \fint_{B_r} \left(\int_{\mathbb{R}^n \backslash B_r} \frac{|v(y, t)|}{|x-y|^{n+2s}} d y \right) |v(x, t)| \psi^2(x) \eta^2(t) d x d t \\
  \quad +C \fint_{\tau_1}^{\tau_2} \fint_{B_r} v^2(x, t) \psi^2(x) \eta(t)\left|\partial_t \eta(t)\right| d x d t+C\fint_{\tau_1}^{\tau_2}\fint_{B_r}v^2(x,t)\left|\nabla \psi (x) \right|^2 \eta^2(t)dxdt.
\end{array}
\end{equation}
Take $(\tau_1,\tau_2)=(t_0-r^2,t_0)\Subset (0,T)$. Now since choices of $\eta$ and $\psi$ are arbitrary, we choose $\eta$ such that $\eta (t_0-r^2)=0$, $\eta\equiv 1$ in $(t_0,t_0-r^2/4)$, $\eta\leq 1$ in $(t_0-r^2,t_0)$ and $\left |\frac{\partial \eta}{\partial t}\right|\leq \frac{C}{r^2}$; and $\psi$ such that $\psi\equiv 1$ inside $B_{r/2}$, $\psi \equiv 0$ outside $B_{3r/4}$, $\psi \leq 1$ in $B_{r}$ and $\left|\nabla \psi\right|\leq \frac{C}{r}$. We get 
\begin{equation}{\nonumber}
     \begin{array}{lcr}
           \fint_{t_0-r^2/4}^{t_0}\int_{B_{r/2}} \fint_{B_{r/2}}\left|u(x,t)-u(y,t)\right|^2 d \mu dt+	\fint_{t_0-r^2/4}^{t_0}\fint_{B_{r/2}} \left|\nabla u(x,t) \right|^2 dx dt+\frac{1}{r^2}\underset{t_0-r^2<t<t_0}{\operatorname{ess} \sup} \fint_{B_r} v^2(x, t) \psi^2(x) d x &&\\ \leq C \fint_{t_0-r^2}^{t_0} \int_{B_r} \fint_{B_r} \max \left\{v(x,t), v(y,t)\right\}^2|\psi(x)-\psi(y)|^2 d \mu dt&&\\
		\quad +C\fint_{t_0-r^2}^{t_0}\fint_{B_r}\left(\int_{\mathbb{R}^n \backslash B_r} \frac{|v(y,t)|}{|x-y|^{n+2s}} d y  \right) |v(x,t)|\psi^2(x)  dxdt&&\\
		\quad+\frac{C}{r^2}\fint_{t_0-r^2}^{t_0}\fint_{B_r}v^2(x,t)\psi^2(x)\eta(t) dxdt+\frac{C}{r^2}\fint_{t_0-r^2}^{t_0}\fint_{B_r}v^2(x,t)\eta^2(t) dxdt.
     \end{array}
\end{equation}
We write the parabolic cylinder $Q_r=B_r\times(t_0-r^2,t_0)$ for convenience. Now we choose $k=(u)_{B_r\times(t_0-r^2,t_0)}=(u)_{Q_r}$ in $v=u-k$, and note that $\psi \leq 1$ in $B_r$ and $\eta\leq 1$ in $(t_0-r^2,t_0)$, to get
\begin{equation}{\label{cacc1}}
     \begin{array}{lcr}
          \quad \fint_{t_0-r^2/4}^{t_0}\int_{B_{r/2}} \fint_{B_{r/2}}\frac{\left|u(x,t)-u(y,t)\right|^2}{\left|x-y\right|^{n+2s}} d xdydt +	\fint_{t_0-r^2/4}^{t_0}\fint_{B_{r/2}} \left|\nabla u(x,t) \right|^2 dxdt&&\\\quad+\frac{1}{r^2}\underset{t_0-r^2<t<t_0}{\operatorname{ess} \sup} \fint_{B_r} \left|u(x,t)-(u)_{Q_r}\right|^2 \psi^2(x)  d x&&\\ \leq C\fint_{t_0-r^2}^{t_0}\int_{B_r} \fint_{B_r} \frac{\max \left\{v(x,t), v(y,t)\right\}^2}{|x-y|^{n+2s}}|\psi(x)-\psi(y)|^2 d x dy dt
           &&\\
		\quad+C   \fint_{t_0-r^2}^{t_0}\fint_{B_r}\left(\int_{\mathbb{R}^n \backslash B_r} \frac{\left|u(y,t)-(u)_{Q_r}\right|}{|x-y|^{n+2s}} d y\right)\left|u(x,t)-(u)_{Q_r}\right|  \psi^2(x)dxdt
		&&\\\quad+\frac{C}{r^2}\fint_{t_0-r^2}^{t_0}\fint_{B_r}\left|u(x,t)-(u)_{Q_r}\right|^2 dxdt.
     \end{array}
 \end{equation} 
 We need to estimate each term present in the RHS to get exact results for our purpose.
\begin{description}[leftmargin=0cm]
\item{\textbf{Estimate for the 1st term:}} Since $|\nabla \psi|\leq C/r$, we get
\begin{equation}{\label{1st}}
    \begin{array}{l}
\quad\fint_{t_0-r^2}^{t_0}\int_{B_r} \fint_{B_r} \frac{\max \left\{v(x,t), v(y,t)\right\}^2}{|x-y|^{n+2s}}|\psi(x)-\psi(y)|^2 d x dy dt\\\leq Cr^{-2}\fint_{t_0-r^2}^{t_0}\int_{B_r} \fint_{B_r} \frac{\left|u(x,t)-(u)_{Q_r}\right|^2}{\left| x-y\right|^{n+2(s-1)}}dxdydt \\\leq Cr^{-2s}\fint_{t_0-r^2}^{t_0}\fint_{B_r} \left|u(x,t)-(u)_{Q_r}\right|^2dxdt.
    \end{array}
\end{equation} 
\item{\textbf{Estimate for the 2nd term:}} $\psi$ is supported inside $B_{3r/4}$ implies $\frac{|y-x_0|}{|x-y|}\leq 4$, so we get by H\"older inequality
\begin{equation*}
\begin{array}{l}
   \quad \fint_{t_0-r^2}^{t_0}\fint_{B_r}\left(\int_{\mathbb{R}^n \backslash B_r} \frac{\left|u(y,t)-(u)_{Q_r}\right|}{|x-y|^{n+2s}} d y\right)\left|u(x,t)-(u)_{Q_r}\right| \psi^2(x) dxdt\\\leq c\left(\fint_{t_0-r^2}^{t_0}\left(\int_{\mathbb{R}^n \backslash B_r} \frac{\left|u(y,t)-(u)_{Q_r}\right|}{|y-x_0|^{n+2s}} d y\right)^2d t\right)^{1/2}\times\left(\fint_{t_0-r^2}^{t_0}\left( \fint_{B_r}\left|u(x,t)-(u)_{Q_r}\right| dx\right)^2dt\right)^{1/2}\\\leq c\left(\int_{t_0-r^2}^{t_0}\left(\int_{\mathbb{R}^n \backslash B_r} \frac{\left|u(y,t)-(u)_{Q_r}\right|}{|y-x_0|^{n+2s}} d y\right)^2d t\right)^{1/2}\times\left(\frac{1}{r^{2}}\fint_{t_0-r^2}^{t_0}\left( \fint_{B_r}\left|u(x,t)-(u)_{Q_r}\right| dx\right)^2dt\right)^{1/2}.
\end{array}
\end{equation*} 
Now, using Young's inequality, we get
\begin{equation}{\label{second}}
    \begin{array}{l}
     \quad   \fint_{t_0-r^2}^{t_0}\fint_{B_r}\left(\int_{\mathbb{R}^n \backslash B_r} \frac{\left|u(y,t)-(u)_{Q_r}\right|}{|x-y|^{n+2s}} d y\right)\left|u(x,t)-(u)_{Q_r}\right| \psi^2(x) dxdt\\\leq C\left(\int_{t_0-r^2}^{t_0}\left(\int_{\mathbb{R}^n \backslash B_r} \frac{\left|u(y,t)-(u)_{Q_r}\right|}{|y-x_0|^{n+2s}} d y\right)^2d t\right)^{1/2}\times\left(\frac{1}{r^{2}}\fint_{t_0-r^2}^{t_0}\fint_{B_r}\left|u(x,t)-(u)_{Q_r}\right|^2 dxdt\right)^{1/2} \\\leq C\int_{t_0-r^2}^{t_0}\left(\int_{\mathbb{R}^n \backslash B_r} \frac{\left|u(y,t)-(u)_{Q_r}\right|}{|y-x_0|^{n+2s}} d y\right)^2d t+\frac{C}{r^{2}}\fint_{t_0-r^2}^{t_0} \fint_{B_r}\left|u(x,t)-(u)_{Q_r}\right|^2  dxdt.
\end{array}
\end{equation}
 \end{description}
So \textbf{the final Caccioppoli estimate} follows from \cref{cacc1}, using \cref{1st,second} that 
\begin{equation}{\label{caccfinal}}
     \begin{array}{lcr}
        \quad   \fint_{t_0-r^2/4}^{t_0}\int_{B_{r/2}} \fint_{B_{r/2}}\frac{\left|u(x,t)-u(y,t)\right|^2}{|x-y|^{n+2s}} d xdydt +	\fint_{t_0-r^2/4}^{t_0}\fint_{B_{r/2}} \left|\nabla u(x,t) \right|^2 dxdt
           &&\\ \leq C\int_{t_0-r^2}^{t_0}\left(\int_{\mathbb{R}^n \backslash B_r} \frac{\left|u(y,t)-(u)_{Q_r}\right|}{|y-x_0|^{n+2s}} d y\right)^2d t +\frac{C}{r^{2s}}\fint_{t_0-r^2}^{t_0}\fint_{B_r}\left|u(x,t)-(u)_{Q_r}\right|^2 dxdt&&\\\quad+\frac{C}{r^2}\fint_{t_0-r^2}^{t_0}\fint_{B_r}\left|u(x,t)-(u)_{Q_r}\right|^2 dxdt,
     \end{array}
 \end{equation}
 where $C\equiv C(n,s)$.
\subsection{{Difference estimate}}
We now prove the difference estimate for $u-w$, where $u$ is solution to our original problem and $w \in L^2(I;u+W_0^{1, 2}(B_{r / 4}))$ is the (unique) solution to
\begin{eqnarray}{\label{heat}}
w_t-\Delta w=0 
&  \text{ in } B_{r/4} \times I, \nonumber\\w=u &  \text{ on } \partial B_{r/4}\times I, \\ w(., \tau_1)=u(., \tau_1) & \text{ on } B_{r/4}\times \{t=\tau_1\}\nonumber,
\end{eqnarray}
where $I=(\tau_1,\tau_2)=(t_0-r^2/16,t_0)\Subset{(0,T)}$ and $B_r\equiv B_r(x_0)\Subset \Omega.$ 
Such a solution always exists as $ u\in C(I;L^2_{\textit{loc}}(\Omega))$ and hence $\underset{t\in I}{\operatorname{sup}}\left|\left|u(.,t)\right|\right|_{L^2(B_{r/4})}< \infty$, which in turn implies that $\left|\left|u(.,\tau_1)\right|\right|_{L^2(B_{r/4})}< \infty$ i.e. $u(.,\tau_1)\in L^2(B_{r/4}).$ The details can be found in [\citealp{LCE}, Chapter-7, Theorem 1,2,3]. So $w$ solves the Euler-Lagrange equation
\begin{equation}{\label{refsol}}
\fint_{t_0-r^2/16}^{t_0} \fint_{B_{r/ 4}(x_0)}( w_{h,t} \varphi +(\nabla w)_h \cdot \nabla \varphi )d x dt=0, \quad \text { for every } \varphi \in W_0^{1, 2}(B_{r / 4}).
\end{equation}
For the rest of this article, we will use the notion for the backward parabolic cylinder as $Q_d=B_d(x_0)\times(t_0-d^2,t_0)$ where $B_d(x_0)\Subset \Omega $ and $(t_0-d^2,t_0)\Subset(0,T)$. The following results are for $w$.\\\\
\textbf{Poincar\'{e}:} The Poincar\'{e} inequality for $w$ will be very useful for our purpose. If we use \cref{Poincar\'{e}} with $\rho \equiv 1$, $I=(t_0-d^2,t_0)$ and for $B_{d}$ where $d\leq r/8$, we will get for our case
\begin{equation*}
 \fint_{t_0-d^2}^{t_0}\fint_{{B}_{d}}\left|{w-(w)_{Q_{d}}}\right|^{2} d x d t \leq C d^2\fint_{t_0-d^2}^{t_0}\fint_{{B}_{d} }|\nabla w|^{2} d xdt+\sup _{t_1, t_2 \in (t_0-d^2,t_0)}\left|(w)_\mu(t_2)-(w)_\mu(t_1)\right|^{2},
\end{equation*}
where $\mu(x) \in C_c^{\infty}({B}_d)$ such that $\int_{{B}_d} \mu(x) d x=1$ with $|\mu| \leq \frac{C(n)}{d^n}$ and $|\nabla \mu| \leq \frac{C(n)}{d^{n+1}}$ and
\begin{equation*}
(w)_\mu(t_i):=\int_{{B}_{d}} w(x, t_i) \mu(x) d x \text { for } i=1,2 .
 \end{equation*}
Now choosing $\mu$ as a test function in \cref{refsol}, for the domain $B_{d}\times(t_1,t_2)$, where $t_1,t_2\in(t_0-d^2, t_0)$, we get
\begin{equation*}
\int_{t_1}^{t_2} \int_{B_{d}}( w_{h,t} \mu +(\nabla w)_h \cdot \nabla \mu )d x dt=0,
\end{equation*}
where $w_h$ is the usual Steklov average.
Now using the fact that $|\nabla \mu| \leq \frac{C(n)}{d^{n+1}}$, we get 
\begin{equation}{\label{w_t}}
    \begin{array}{l}
     \quad\quad   \int_{t_1}^{t_2} \int_{B_{d}}w_{h,t} \mu (x)  dxdt=-\int_{t_1}^{t_2} \int_{B_{d}}(\nabla w)_h \cdot \nabla \mu d x dt \\ \implies \int_{B_{d}}w_{h} (x,t_2)\mu(x)  d x-\int_{B_{d}}w_{h} (x,t_1)\mu(x)  dx=-\int_{t_1}^{t_2} \int_{B_{d}}(\nabla w)_h \nabla \mu  d x dt \\ \implies \left|\int_{B_{d}}w_{h} (x,t_2)\mu(x)  d x-\int_{B_{d}}w_{h} (x,t_1)\mu(x)  dx\right|\leq \int_{t_1}^{t_2} \int_{B_{d}}|(\nabla w)_h| |\nabla \mu | d x dt
       \\
\implies \left|(w_h)_\mu(t_2)-(w_h)_\mu(t_1\right)|\leq \frac{C}{d^{n+1}} \int_{t_0-d^2}^{t_0} \int_{B_{d}}\left|(\nabla w)_h\right|dxdt \\ \implies \sup _{t_1, t_2 \in (t_0-d^2,t_0)}\left|(w_h)_\mu(t_2)-(w_h)_\mu(t_1\right)|^{2} \leq Cd^2 \left(\fint_{t_0-d^2}^{t_0} \fint_{B_{d}}\left|(\nabla w)_h\right|dxdt \right)^2 \\ \implies \sup _{t_1, t_2 \in (t_0-d^2,t_0)}\left|(w_h)_\mu(t_2)-(w_h)_\mu(t_1)\right|^{2} \leq Cd^2\fint_{t_0-d^2}^{t_0} \fint_{B_{d}}\left|(\nabla w)_h\right|^2dxdt.
    \end{array}
\end{equation}
Using these, we get for each $d\leq r/8$,
\begin{equation}{\label{poin}}
    \fint_{t_0-d^2}^{t_0}\fint_{{B}_{d}}\left|{w-(w)_{Q_d}}\right|^{2} d xdt \leq C d^2\fint_{t_0-d^2}^{t_0}\fint_{{B}_{d} }|\nabla w|^{2} d xdt.
\end{equation}
\textbf{Regularity:} The following regularity estimates are followed by $w$. For details, we refer [\citealp{GaryM}, Chapter 4, Lemma 4.5]
\begin{equation}{\label{regw}}
    \fiint_{Q_d} w^2(x,t)dxdt\leq c\fiint_{Q_{r/8}} w^2(x,t)dxdt 
    \end{equation}
    and 
    \begin{equation}{\label{reguw}}
    \fiint_{Q_d} |w(x,t)-(w)_{Q_d}|^2dxdt\leq c\left(\frac{d}{r}\right)^2\fiint_{Q_{r/8}} |w(x,t)-(w)_{Q_{r/8}}|^2dxdt,
\end{equation}
for all $d\in(0,r/8]$, where $c$ is a constant depending only on $n$.
\begin{theorem}{\label{diffest}}[Energy estimate]
	\label{theorem3.2}
	Let $w \in L^2(I;u+W_0^{1, 2}(B_{r / 4}))$ with $r\leq 1$ be as in \cref{heat}. Then the following
	\begin{equation*}
	\begin{array}{l}
		\quad\fint_{t_0-r^2/16}^{t_0}\fint_{B_{r / 4}(x_0)}|\nabla (u(x,t)-w(x,t))|^2 d x dt +\frac{8}{r^2}
  \fint_{B_{r/4}(x_0)} (u-w)^2(x,t_0)dx\\\leq C\bigg[\int_{t_0-r^2}^{t_0}\left(\int_{\mathbb{R}^n \backslash B_r(x_0)} \frac{\left|u(y,t)-(u)_{Q_r}\right|}{|y-x_0|^{n+2s}} d y\right)^2d t\\\quad+ r^{2-2s}\fint_{t_0-r^2/4}^{t_0}\int_{B_{r/2}(x_0)}\fint_{B_{r/2}(x_0)}\frac{\left| u(x,t)-u(y,t)\right|^2}{| x-y|^{n+2s}}dxdydt\\\quad+r^{2-2s}r^{-2s}\fint_{t_0-r^2}^{t_0}\fint_{B_r\left(x_0\right)}\left|u(x,t)-k\right|^2 d x d t\bigg]
	\end{array}
	\end{equation*}
	holds for $k\in \mathbb{R}$, where $C \equiv C(n,s).$ 
\end{theorem}
\begin{proof}
We extend $w\equiv u$ outside $B_{r/4}$, and see that $(u-w)_h\in L^2(I;W^{1,2}_0(B_{r/4}))$. Then \cref{embedding2} implies that $(u-w)_h\in L^2(I;W^{s,2}_0(B_{r/4}))$, and since $(u-w)\equiv 0$ outside $B_{r/4}$, so $(u-w)_h\in L^2(I;W^{s,2}(\mathbb{R}^n))$. Further $(u-w)_h\in W^{1,2}(I ; L^2(B_{r}))$. So, we can use $(u-w)_h$ as a test function in both cases. We have
\begin{equation}
	\label{differ}
	\begin{array}{l}
		 \fint_{t_0-r^2/16}^{t_0} \fint_{B_{r / 4}}\left|\nabla u_h-\nabla w_h\right|^2 dx dt
   =\underbrace{-\fint_{t_0-r^2/16}^{t_0} \fint_{B_{r /4}}u_{h,t} (u-w)_h dx dt +\fint_{t_0-r^2/16}^{t_0} \fint_{B_{r / 4}}w_{h,t} (u-w)_h dx dt}_{(I)_h}
	\\\qquad\qquad\qquad\qquad\qquad\qquad\qquad
		  \underbrace{- \frac{c}{2}\fint_{t_0-r^2/16}^{t_0}\int_{B_{r/2}} \fint_{B_{r/2 }}(u(x,t)-u(y,t))_h((u-w)(x,t)-(u-w)(y,t))_h d\mu dt}_{(II)_h} \\\qquad\qquad\qquad\qquad\qquad\qquad\qquad
		 \underbrace{ -c \fint_{t_0-r^2/16}^{t_0}\int_{\mathbb{R}^n \backslash B_{r/2 }}  \fint_{B_{r /2}}(u(x,t)-u(y,t))_h (u-w)_h(x,t) \times K(x, y, t) d x  dy dt}_{(III)_h}.
	\end{array}
\end{equation}
\begin{description}[leftmargin=0cm]
    \item{\textbf{Estimate for $(I)_h$:}} Since $u\equiv w$ outside $B_{r/4}$ and $(u-w)(x,t_0-r^2/16)=0$,
\begin{equation}{\label{diff1}}
    \begin{array}{rcl}
         (\mathrm{I})_h
         &=&-\fint_{t_0-r^2/16}^{t_0} \fint_{B_{r / 4}}(u-w)_{h,t} (u-w)_h dx dt=-\frac{1}{2}\fint_{t_0-r^2/16}^{t_0} \fint_{B_{r / 4}}\partial _t (u-w)_h^2(x,t) dx dt\\&=&-\frac{1}{2r^2/16} \fint_{B_{r/4}} (u-w)_h^2(x,t_0)dx.
    \end{array}
\end{equation}
\item{\textbf{Estimate for $(II)_h$:}} H\"older inequality and embedding inequality of \cref{fracemb} yields
\begin{equation*}
    \begin{array}{l}
    |(\mathrm{II})_h| \leq C\left(\fint_{t_0-r^2/4}^{t_0} \int_{B_{r/2}} \fint_{B_{r/2 }} \frac{\left|\left(u(x,t)-u(y,t)\right)_h\right|^2}{|x-y|^{n+2s}} dxdydt\right)^{1/2} \quad\quad\quad\quad\quad\quad\quad\quad\quad\quad\quad\quad\quad\quad\quad\quad\quad\quad\quad\quad\quad\quad\quad\quad\quad\quad\quad\quad\quad\quad
\\\quad\quad\quad\quad\times\left(\fint_{t_0-r^2/16}^{t_0} \int_{B_{r/2}} \fint_{B_{r / 2}}\frac{\left|\left((u-w)(x,t)-(u-w)(y,t)\right)_h\right|^2}{|x-y|^{n+2s}} dxdydt\right)^{1 / 2} 
\end{array}
\end{equation*}
$$
\begin{aligned}
&\leq Cr^{(1-s)}\left(\fint_{t_0-r^2/4}^{t_0} \int_{B_{r/2}} \fint_{B_{r /2}} \frac{\left|\left(u(x,t)-u(y,t)\right)_h\right|^2}{|x-y|^{n+2s}} dxdydt\right)^{1/2}
\left(\fint_{t_0-r^2/16}^{t_0} \fint_{B_{r/2}} \left|\nabla(u-w)_h\right|^2 dxdt\right)^{1 / 2}\\
& \leq Cr^{(1-s)}\left(\fint_{t_0-r^2/4}^{t_0} \int_{B_{r/2}} \fint_{B_{r /2}} \frac{\left|\left(u(x,t)-u(y,t)\right)_h\right|^2}{|x-y|^{n+2s}} dxdydt\right)^{1/2}
\left(\fint_{t_0-r^2/16}^{t_0} \fint_{B_{r/4}} \left|\nabla(u-w)_h\right|^2 dxdt\right)^{1 / 2},
\end{aligned}
$$
with $C \equiv C(n,s)$. We have used the fact that $u-w\equiv 0$ outside $B_{r/4}$ in the last line.\\ Now by Young's inequality with $\varepsilon$ we get
\begin{equation}{\label{diff2}}
\begin{array}{rcl}
|(\mathrm{II})_h| &\leq& Cr^{2(1-s)}\fint_{t_0-r^2/4}^{t_0}\int_{B_{r/2}} \fint_{B_{r/2 }} \frac{\left|\left(u(x,t)-u(y,t)\right)_h\right|^2}{|x-y|^{n+2s}} dxdydt\\&&+\varepsilon \fint_{t_0-r^2/16}^{t_0} \fint_{B_{r/4}} \left|\nabla(u-w)_h\right|^2 dxdt.
\end{array}
\end{equation}
\item{\textbf{Estimate for $(III)_h$:}} Since $x \in B_{r / 4}$ and $ y \in \mathbb{R}^n \backslash B_{r /2}$, we can see that $\frac{|y-x_0|} {|x-y|} \leq 2$ and as $w$ is supported in $B_{r / 4}$, we estimate
\begin{equation*}
\begin{array}{rcl}
 {|(\mathrm{III})_h|} &\leq& C\fint_{t_0-r^2/16}^{t_0}\int_{\mathbb{R}^n \backslash B_{r/2 }}  \fint_{B_{r /2}}\frac{|((u(x,t)-k)_h-(u(y,t)-k)_h) (u-w)_h(x,t)| }{|x-y|^{n+2s}} d x  dy dt\\& \leq &C \fint_{t_0-r^2/16}^{t_0} \int_{\mathbb{R}^n \backslash B_{r /2}}  \fint_{B_{r/2 }}  \frac{\max {\left\{ |u_h(x,t)-k|,|u_h(y,t)-k| \right \} }|(u-w)_h(x,t)|}{|x-y|^{n+2s }} d x dy dt \\& \leq &C \fint_{t_0-r^2/16}^{t_0} \int_{\mathbb{R}^n \backslash B_{r /2}}  \fint_{B_{r/2 }}  \frac{\max {\left\{ |u_h(x,t)-k|,|u_h(y,t)-k| \right \} }|(u-w)_h(x,t)|}{|y-x_0|^{n+2s }} d x dy dt.
 \end{array}
 \end{equation*}
 Now, separating the integrals with respect to $x$ and $y$, we get
\begin{equation*}
\begin{array}{rcl}
 {|(\mathrm{III})_h|}& \leq &C r^{-2s } \fint_{t_0-r^2/16}^{t_0} \fint_{B_{r /2}}|u_h(x,t)-k||(u-w)_h(x,t)| d x d t\\
&& +C\fint_{t_0-r^2/16}^{t_0}\left( \int_{\mathbb{R}^n \backslash B_{r /2}} \frac{|u_h(y,t)-k|}{|y-x_0|^{n+2s }} d y \right)\times\left(\fint_{B_{r/4 }}|(u-w)_h(x,t)| dx \right)dt\\
& \leq& C r^{-2s } \fint_{t_0-r^2/16}^{t_0} \fint_{B_{r /2}}|u_h(x,t)-k||(u-w)_h(x,t)| d x dt\\
&& +C\fint_{t_0-r^2/16}^{t_0}\left( \int_{\mathbb{R}^n \backslash B_{r }} \frac{|u_h(y,t)-k|}{|y-x_0|^{n+2s }} d y \right)\times\left(\fint_{B_{r/4 }}|(u-w)_h(x,t)| dx \right)dt\\&&+C\fint_{t_0-r^2/16}^{t_0}\left( \int_{B_r \backslash B_{r /2}} \frac{|u_h(y,t)-k|}{|y-x_0|^{n+2s }} d y \right)\times\left(\fint_{B_{r /4}}|(u-w)_h(x,t)| dx\right) dt\\
& \leq& C r^{-2s } \fint_{t_0-r^2/16}^{t_0} \fint_{B_{r /2}}|u_h(x,t)-k||(u-w)_h(x,t)| d x dt\\
& &+C\fint_{t_0-r^2/16}^{t_0}\left( \int_{\mathbb{R}^n \backslash B_{r }} \frac{|u_h(y,t)-k|}{|y-x_0|^{n+2s }} d y \right)\times\left(\fint_{B_{r/4 }}|(u-w)_h(x,t)| dx \right)dt\\&&+Cr^{-2s}\fint_{t_0-r^2/16}^{t_0}\left( \fint_{B_r } |u_h(y,t)-k| d y \right)\times\left(\fint_{B_{r /4}}|(u-w)_h(x,t)| dx\right) dt,
\end{array}
\end{equation*}
for $C \equiv C(n,s)$. Using H\"older inequality, we get
$$
\begin{aligned}
{|(\mathrm{III})_h|} 
 \leq &C\Bigg[r^{-2s } \left(\fint_{t_0-r^2/16}^{t_0} \fint_{B_{r /2}}|u_h(x,t)-k|^2dxdt\right)^{1/2}\times\left(\fint_{t_0-r^2/16}^{t_0} \fint_{B_{r/4 }}|(u-w)_h(x,t)|^2 dxdt\right)^{1 /2}\\&+\left(\fint_{t_0-r^2/16}^{t_0}\left( \int_{\mathbb{R}^n \backslash B_{r }} \frac{|u_h(y,t)-k|}{|y-x_0|^{n+2s }} d y\right)^2 dt\right)^{1/2}\times\left(\fint_{t_0-r^2/16}^{t_0}\left(\fint_{B_{r/4 }}|(u-w)_h(x,t)| dx\right)^2 dt\right)^{1/2}\\&+r^{-2s}\left(\fint_{t_0-r^2/16}^{t_0}\left( \fint_{B_r } |u_h(y,t)-k| d y \right)^2dt\right)^{1/2}\times\left(\fint_{t_0-r^2/16}^{t_0}\left(\fint_{B_{r /4}}|(u-w)_h(x,t)| dx\right)^2 dt\right)^{1/2}\Bigg]
\end{aligned}
$$
$$
\begin{aligned}
 \leq& C\Bigg[r^{-2s } \left(\fint_{t_0-r^2}^{t_0} \fint_{B_{r }}|u_h(x,t)-k|^2dxdt\right)^{1/2}\times\left(\fint_{t_0-r^2/16}^{t_0}\fint_{B_{r/4 }}|(u-w)_h(x,t)|^2 dxdt\right)^{1 /2}\\&+\left(r^{-2}\int_{t_0-r^2}^{t_0}\left( \int_{\mathbb{R}^n \backslash B_{r }} \frac{|u_h(y,t)-k|}{|y-x_0|^{n+2s }} d y\right)^2 dt\right)^{1/2}\times\left(\fint_{t_0-r^2/16}^{t_0}\left(\fint_{B_{r/4 }}|(u-w)_h(x,t)| dx\right)^2 dt\right)^{1/2}\Bigg]\\
 \leq& Cr^{-s}\Bigg[r^{-s} \left(\fint_{t_0-r^2}^{t_0} \fint_{B_{r }}|u_h(x,t)-k|^2dxdt\right)^{1/2}+\left(r^{-2+2s}\int_{t_0-r^2}^{t_0}\left( \int_{\mathbb{R}^n \backslash B_{r }} \frac{|u_h(y,t)-k|}{|y-x_0|^{n+2s }} d y\right)^2 dt\right)^{1/2}\Bigg]\\&\times\left(\fint_{t_0-r^2/16}^{t_0}\fint_{B_{r/4 }}|(u-w)_h(x,t)|^2 dxdt\right)^{1 /2}.
\end{aligned}
$$
Now we observe that $(u-w)_h\equiv 0$ on and outside $B_{r/4}$. This observation, along with Poincar\'{e} inequality gives us the following
\begin{equation*}
\fint_{t_0-r^2/16}^{t_0}\fint_{B_{r /4}}|(u-w)_h(x,t)|^2 dxdt
 \leq  C r^{2}\fint_{t_0-r^2/16}^{t_0} \fint_{B_{r /4}}|\nabla (u-w)_h|^2 dxdt,
\end{equation*}
where $C\equiv C(n)$.
So we get 
$$
\begin{aligned}
{|(\mathrm{III})_h|} 
 \leq& Cr^{1-s}\Bigg[r^{-s} \left(\fint_{t_0-r^2}^{t_0} \fint_{B_{r }}|u_h(x,t)-k|^2dxdt\right)^{1/2}+\left(r^{-2+2s}\int_{t_0-r^2}^{t_0}\left( \int_{\mathbb{R}^n \backslash B_{r }} \frac{|u_h(y,t)-k|}{\left|y-x_0\right|^{n+2s }} d y\right)^2 dt\right)^{1/2}\Bigg]\\&\times\left(\fint_{t_0-r^2/16}^{t_0}\fint_{B_{r/4 }}|\nabla(u-w)_h(x,t)|^2 dxdt\right)^{1 /2}.
\end{aligned}
$$
Now using Young's inequality with $\varepsilon$, we get
\begin{equation}{\label{diff3}}
	\begin{array}{rcl}
    {|(\mathrm{III})_h|}&\leq& Cr^{2(1-s) }\Bigg[r^{-2s}\fint_{t_0-r^2}^{t_0} \fint_{B_{r }}|u_h(x,t)-k|^2dxdt+r^{-2+2s}\int_{t_0-r^2}^{t_0}\left( \int_{\mathbb{R}^n \backslash B_{r }} \frac{|u_h(y,t)-k|}{\left|y-x_0\right|^{n+2s }} d y\right)^2 dt\Bigg]\\&&+\varepsilon\fint_{t_0-r^2/16}^{t_0}\fint_{B_{r/4 }}|\nabla(u-w)_h(x,t)|^2 dxdt.\end{array}
\end{equation}
\end{description}
Combining \cref{diff1,diff2,diff3} and choosing $\varepsilon=\frac{1}{4}$, from \cref{differ} we get
\begin{multline*}
\fint_{t_0-r^2/16}^{t_0}\fint_{B_{r / 4}(x_0)}|\nabla (u(x,t)-w(x,t))_h|^2 d x dt +\frac{8}{r^2}\fint_{B_{r/4}(x_0)}\left | (u-w)_h^2(x,t_0)\right| dx\\
\leq Cr^{1-s}\bigg[r^{(-1+s)}\int_{t_0-r^2}^{t_0}\left( \int_{\mathbb{R}^n \backslash B_r (x_0)} \frac{|u_h(y,t)-k|}{\left|y-x_0\right|^{n+2s }} d y\right)^2 dt
+r^{(1-s)}r^{-2s}\fint_{t_0-r^2}^{t_0} \fint_{B_r\left(x_0\right)}\left|u_h(x,t)-k\right|^2 d x d t\\+ r^{(1-s)}\fint_{t_0-r^2/4}^{t_0}\int_{B_{r/2}(x_0)}\fint_{B_{r/2}(x_0)}\frac{\left| \left(u(x,t)-u(y,t)\right)_h\right|^2}{\left| x-y\right|^{n+2s}}dxdydt\bigg],
\end{multline*}
where $C\equiv C(n,s)$.\\
Now using convergence properties of $u_h$ and \cref{steklov}, we finally get
\begin{multline}{\label{diffinal}}
\fint_{t_0-r^2/16}^{t_0}\fint_{B_{r / 4}(x_0)}|\nabla (u(x,t)-w(x,t))|^2 d x dt +\frac{8}{r^2}\fint_{B_{r/4}(x_0)}\left | (u-w)^2(x,t_0)\right| dx\\
\leq Cr^{1-s}\bigg[r^{(-1+s)}\int_{t_0-r^2}^{t_0}\left( \int_{\mathbb{R}^n \backslash B_r (x_0)} \frac{|u(y,t)-k|}{\left|y-x_0\right|^{n+2s }} d y\right)^2 dt
+r^{(1-s)}r^{-2s}\fint_{t_0-r^2}^{t_0} \fint_{B_r\left(x_0\right)}\left|u(x,t)-k\right|^2 d x d t\\+ r^{(1-s)}\fint_{t_0-r^2/4}^{t_0}\int_{B_{r/2}(x_0)}\fint_{B_{r/2}(x_0)}\frac{\left| \left(u(x,t)-u(y,t)\right)\right|^2}{\left| x-y\right|^{n+2s}}dxdydt\bigg]
\end{multline}
holds with $C \equiv C(n,s).$
\end{proof}
\subsection{Merging Caccioppoli and difference estimates}{\label{varepsilon}}We will now combine the Caccioppoli inequality with the difference estimate. For that we use $k=(u)_{Q_r}$ in \cref{diffinal} and note that for $r\leq 1$, $r^{-2s}\leq r^{-2}$, so using \cref{caccfinal} we get 
\begin{multline}{\label{DIFFER}}
\fint_{t_0-r^2/16}^{t_0}\fint_{B_{r / 4}(x_0)}|\nabla (u(x,t)-w(x,t))|^2 d x dt\leq Cr^{1-s}\bigg[r^{(-1+s)}\int_{t_0-r^2}^{t_0}\left( \int_{\mathbb{R}^n \backslash B_r (x_0)} \frac{|u(y,t)-(u)_{Q_r}|}{\left|y-x_0\right|^{n+2s }} d y\right)^2 dt
\\\qquad\qquad\qquad\qquad\qquad\qquad\qquad\qquad\qquad\qquad\qquad\qquad+r^{-2}\fint_{t_0-r^2}^{t_0} \fint_{B_r\left(x_0\right)}\left|u(x,t)-(u)_{Q_r}\right|^2 d x d t\bigg]\\\qquad\qquad\qquad\qquad\qquad\qquad\qquad\quad\quad\qquad
= Cr^{\varepsilon}\bigg[r^{(-\varepsilon)}\int_{t_0-r^2}^{t_0}\left( \int_{\mathbb{R}^n \backslash B_r (x_0)} \frac{|u(y,t)-(u)_{Q_r}|}{\left|y-x_0\right|^{n+2s }} d y\right)^2 dt\\
\qquad\qquad\qquad\qquad\qquad\qquad\qquad\qquad\qquad\qquad\qquad\qquad+r^{-2+1-s-\varepsilon}\fint_{t_0-r^2}^{t_0} \fint_{B_r\left(x_0\right)}\left|u(x,t)-(u)_{Q_r}\right|^2 d x d t\bigg],
\end{multline}
where $\varepsilon<1-s$ is a very small positive number. Further restrictions on $\varepsilon$ will be determined later. 
Using Poincar\'{e}, we now get the \textbf{final energy estimate} as 
\begin{equation}{\label{DIFFERENCE}}
\begin{array}{l}
\quad\fint_{t_0-r^2/16}^{t_0}\fint_{B_{r / 4}(x_0)}| \left(u(x,t)-w(x,t)\right)|^2 d x dt\\\leq Cr^2\fint_{t_0-r^2/16}^{t_0}\fint_{B_{r / 4}(x_0)}|\nabla (u(x,t)-w(x,t))|^2 d x dt\\\leq Cr^{\varepsilon}\bigg[r^{2-\varepsilon}\int_{t_0-r^2}^{t_0}\left( \int_{\mathbb{R}^n \backslash B_r (x_0)} \frac{|u(y,t)-(u)_{Q_r}|}{\left|y-x_0\right|^{n+2s }} d y\right)^2 dt
+\fint_{t_0-r^2}^{t_0} \fint_{B_r\left(x_0\right)}\left|u(x,t)-(u)_{Q_r}\right|^2 d x d t\bigg].
\end{array}
\end{equation}
\subsection{Decay estimate for the non-local integral}
We now see the following important results that will be helpful later.  
\begin{lemma}{\label{decay}} Let $B_d(x_0) \subset B_{r}(x_0)\Subset\Omega$ be two concentric balls, and $g \in  L^{2}_{\textit{loc}}(0,T; W_{\textit{loc}}^{1,2}(\Omega))
\cap L^{\infty}_{\textit{loc}}(0,T ; L_{2s}^{1}(\mathbb{R}^n))
$. Let $(t_0-d^2,t_0)\subset(t_0-r^2,t_0)\Subset(0,T)$; and $Q_d\subset Q_r$ are the corresponding parabolic cylinders. If $d>0$ and $\theta \in(0,1)$ are such that $\theta r \leq d \leq r$, then
\begin{equation}{\label{av}}
\begin{array}{rcl}
   \quad \left(\fiint_{Q_d}\left|g(x,t)-(g)_{Q_d}\right|^2 d x d t\right)^{1/2}&\leq& 2 \theta ^{-1-n/2}\left(\fiint_{Q_r}\left|g(x,t)-(g)_{Q_r}\right|^2 d x d t\right)^{1/2}.
   \end{array}
\end{equation}
Whenever $0<d<r \leq 1$, it holds that
\begin{equation}{\label{snail}}
\begin{array}{l}
\quad  \left(d^{2-\varepsilon}\int_{t_0-d^2}^{t_0}\left( \int_{\mathbb{R}^n \backslash B_{d }} \frac{|g(y,t)-(g)_{Q_d}|}{\left|y-x_0\right|^{n+2s }} d y\right)^2 dt\right)^{1/2}  \\\leq  c\left(\frac{d}{r}\right)^{{(2-\varepsilon)/2}} \left(r^{(2-\varepsilon)}\int_{t_0-r^2}^{t_0}\left( \int_{\mathbb{R}^n \backslash B_{r}} \frac{|g(y,t)-(g)_{Q_r}|}{\left|y-x_0\right|^{n+2s }} d y\right)^2 dt\right)^{1/2} \\
\quad+ d^{(2-\varepsilon) / 2-s} \Bigg[c\int_d^{r}\left(\frac{d}{v}\right)^s \left(\fiint_{Q_v}\left|g(x,t)-(g)_{Q_v}\right|^2 d x d t\right)^{1/2} \frac{d v}{v} 
\\\quad+c\left(\frac{d}{r}\right)^s \left(\fiint_{Q_r}\left|g(x,t)-(g)_{Q_r}\right|^2 d x d t\right)^{1/2}\Bigg],
\end{array}
\end{equation}
with $c \equiv c(n, s)$, and $\varepsilon$ is as in \cref{varepsilon}.
\end{lemma}
\begin{proof} We consider all the balls to be centred at $x_0$. Let us now recall the following property
\begin{equation}{\label{av2}}
\left(\fiint_{Q_\varrho}\left|g-(g)_{Q_\varrho}\right|^2 d x dt\right)^{1 / 2} \leq 2\left(\fiint_{Q_\varrho}|g-k|^2 d x\right)^{1 / 2} ,   
\end{equation}
that holds for all $k \in \mathbb{R}$; from this choosing $k=(u)_{Q_r}$, we get \cref{av} i.e.
\begin{equation*}
\begin{array}{rcl}
   \quad \left(\fiint_{Q_d}\left|g(x,t)-(g)_{Q_d}\right|^2 d x d t\right)^{1/2}&\leq& 2 \theta ^{-1-n/2}\left(\fiint_{Q_r}\left|g(x,t)-(g)_{Q_r}\right|^2 d x d t\right)^{1/2}.
   \end{array}
\end{equation*}
For the proof of \cref{snail}, let $B_d \subset B_{r}$, we then split
\begin{equation}{\label{snail1}}
\begin{array}{l}
\quad \left( d^{(2-\varepsilon)}\int_{t_0-d^2}^{t_0} \left(\int_{\mathbb{R}^n \backslash B_{d }} \frac{|g(y,t)-(g)_{Q_d}|}{\left|y-x_0\right|^{n+2s }} d y \right)^2dt \right)^{1/2} \\\leq c\Bigg[ \left(\frac{d}{r}\right)^{(2-\varepsilon)/2} \left(r^{(2-\varepsilon)}\int_{t_0-r^2}^{t_0}\left(\int_{\mathbb{R}^n \backslash B_{r}} \frac{|g(y,t)-(g)_{Q_r}|}{\left|y-x_0\right|^{n+2s }} d y \right)^2dt \right)^{1/2}\\\quad+\underbrace{d^{(2-\varepsilon)/2-s}\left(\frac{d}{r}\right)^s\left|(g)_{Q_d}-(g)_{Q_r}\right|}_{T_1} 
\\\quad+\underbrace{\left(d^{(2-\varepsilon)}\int_{t_0-d^2}^{t_0}\left(\int_{B_{r} \backslash B_d} \frac{\left|g(x,t)-(g)_{Q_d}\right|}{\left|x-x_0\right|^{n+2s }} d x\right)^2dt \right)^{1 / 2}}_{T_2}\Bigg],
\end{array}
\end{equation}
where $c \equiv c(n, s)$. We have used
\begin{equation*}
d \lambda_{x_0}(\mathbb{R}^n \backslash B_d)=c d^{-2s}, \quad d \lambda_{x_0}(x):=\frac{d x}{\left|x-x_0\right|^{n+2s }}.
\end{equation*} 
 If $r / 4 \leq d<r$, then from \cref{snail1}, using \cref{av,av2} , we get that 
 \begin{equation*}
 T_1+T_2 \leq c d^{(2-\varepsilon)/2-s}\left(\frac{d}{r}\right)^s 
 \left(\fiint_{Q_r}\left|g(x,t)-(g)_{Q_r}\right|^2 d x d t\right)^{1/2}
 \end{equation*}
 holds with $c \equiv c(n, s)$. Therefore, we can assume that $d<r / 4$. That implies there exists $\lambda \in(1 / 4,1 / 2)$ and $\kappa \in \mathbb{N}, \kappa \geq 2$ so that $d=\lambda^\kappa r$. Note that $\lambda$ and $\frac{1}{\lambda}$ both are bounded.
\begin{description}[leftmargin=0cm]
    \item{\textbf{Estimate for $T_1$:}} Using triangle and H\"older's inequalities, along with \cref{av,av2} repeatedly, we estimate $T_1$ as
      \begin{equation}{\label{T_1}}
\begin{array}{l}
T_1=d^{(2-\varepsilon)/2-s}\left(\frac{d}{r}\right)^s\left|(g)_{Q_d}-(g)_{Q_r}\right| \\\quad\leq d^{(2-\varepsilon)/2-s}\left(\frac{d}{r}\right)^s\left|(g)_{Q_{\lambda r}}-(g)_{Q_{r}}\right|+d^{(2-\varepsilon)/2-s}\left(\frac{d}{r}\right)^s\left|(g)_{Q_{\lambda r}}-(g)_{Q_{\lambda^\kappa r}}\right|\\\quad\leq c d^{(2-\varepsilon)/2-s}\left(\frac{d}{r}\right)^s \left(\fiint_{Q_r}\left|g(x,t)-(g)_{Q_r}\right|^2 d x d t\right)^{1/2}+d^{(2-\varepsilon)/2-s}\left(\frac{d}{r}\right)^s \sum_{i=1}^{\kappa-1}\left|(g)_{Q_{\lambda^{i} r}}-(g)_{Q_{\lambda ^{i+1}{r}}}\right|\\
\quad \leq c d^{(2-\varepsilon)/2-s}\left(\frac{d}{r}\right)^s \left(\fiint_{Q_r}\left|g(x,t)-(g)_{Q_r}\right|^2 d x d t\right)^{1/2}\\\quad\quad+c d^{(2-\varepsilon)/2-s}\left(\frac{d}{r}\right)^s \sum_{i=1}^{\kappa}\left(\fiint_{Q_{\lambda^i r}}\left|g(x,t)-(g)_{Q_{\lambda^{i} r}}\right|^2
 d xd t\right)^{1 / 2} \\
 \quad\leq c d^{(2-\varepsilon)/2-s}\left(\frac{d}{r}\right)^s \left(\fiint_{Q_r}\left|g(x,t)-(g)_{Q_r}\right|^2 d x d t\right)^{1/2}\\\quad\quad+cd^{(2-\varepsilon)/2-s}\left(\frac{d}{r}\right)^s\sum_{i=1}^{\kappa} \int_{\lambda^i r}^{\lambda^{i-1} r} \left(\fiint_{Q_{\lambda^ir}}\left|g(x,t)-(g)_{Q_{\lambda^ir}}\right|^2 d x d t\right)^{1/2}\frac{d v}{v} 
\\ \quad\leq c d^{(2-\varepsilon)/2-s}\left(\frac{d}{r}\right)^s \left(\fiint_{Q_r}\left|g(x,t)-(g)_{Q_r}\right|^2 d x d t\right)^{1/2}\\\quad\quad+cd^{(2-\varepsilon)/2-s}\left(\frac{d}{r}\right)^s\sum_{i=1}^{\kappa} \int_{\lambda^i r}^{\lambda^{i-1} r} \left(\fiint_{Q_v}\left|g(x,t)-(g)_{Q_v}\right|^2 d x d t\right)^{1/2} \frac{d v}{v} \\
\quad\leq c d^{(2-\varepsilon)/2-s}\left(\frac{d}{r}\right)^s \left(\fiint_{Q_r}\left|g(x,t)-(g)_{Q_r}\right|^2 d x d t\right)^{1/2}\\\quad\quad+c d^{(2-\varepsilon)/2-s}\left(\frac{d}{r}\right)^s \int_d^{r} \left(\fiint_{Q_v}\left|g(x,t)-(g)_{Q_v}\right|^2 d x d t\right)^{1/2} \frac{d v}{v} \text {, } 
\end{array}    
\end{equation}
with $c \equiv c(n, s)$.
\item{\textbf{Estimate for $T_2$:}} Rewriting $r=\lambda^{-\kappa} d$ and using Jensen's inequality we get
\begin{equation}{\label{jensen}}
\begin{array}{rcl}
\left((\lambda^{-i}d)^{-2s}\int_{t_0-d^2}^{t_0}\fint_{B_{\lambda ^{-i}d}}\left|g(x,t)-(g)_{Q_d}\right|^2d x dt\right)^{1 / 2}&\leq& \left(\fint_{t_0-(\lambda^{-i}d)^2}^{t_0}\fint_{B_{\lambda ^{-i}d}}\left|g(x,t)-(g)_{Q_d}\right|^2d x dt\right)^{1 / 2}\\&\leq &\left(\fiint_{Q_{\lambda^{-i}d}}\left|g(x,t)-(g)_{Q_{\lambda^{-i}d}}\right|^2 d x d t\right)^{1/2}\smallskip\\&&+\left|(g)_{Q_{\lambda^{-i}d}}-(g)_{Q_{d}}\right|\\  &\leq& c \sum_{m=0}^i \left(\fiint_{Q_{\lambda^{-m}d}}\left|g(x,t)-(g)_{Q_{\lambda^{-m}d}}\right|^2 d x d t\right)^{1/2} ,
\end{array}
\end{equation}
for $0 \leq i \leq k$. We use \cref{av,av2,jensen} to estimate $T_2$. By the discrete Fubini theorem, we obtain
\begin{equation*}
\begin{array}{rcl} 
T_2&=&\left(d^{(2-\varepsilon)}\int_{t_0-d^2}^{t_0}\left(\int_{B_{r} \backslash B_d} \frac{\left|g(x,t)-(g)_{Q_d}\right|}{\left|x-x_0\right|^{n+2s }} d x\right)^2dt \right)^{1 / 2}\\ &\leq&  c d^{(2-\varepsilon)/2-s} \left(d^{-2s}\int_{t_0-d^2}^{t_0}\left(\sum_{i=0}^{\kappa-1}\lambda^{2i s }\left(\lambda^{-i} d\right)^{-n} \int_{B_{\lambda^{-i-1}d} \backslash B_{\lambda^{-i}d}} \mid g(x,t)-(g)_{Q_d}\mid  d x\right)^2d t\right)^{1 / 2}\\ &\leq&  c d^{(2-\varepsilon)/2-s} \left(d^{-2s}\int_{t_0-d^2}^{t_0}\left(\sum_{i=0}^{\kappa}\lambda^{2i s } \fint_{B_{\lambda^{-i}d}} \mid g(x,t)-(g)_{Q_d}\mid  d x\right)^2d t\right)^{1 / 2}.
\end{array}
\end{equation*}
Now, using Minkowski inequality, we get
\begin{equation*}
    \begin{array}{rcl}
    T_2&\leq&  c d^{(2-\varepsilon)/2-s} \sum_{i=0}^{\kappa}\left(d^{-2s}\int_{t_0-d^2}^{t_0}\left(\lambda^{2i s } \fint_{B_{\lambda^{-i}d}}\left| g(x,t)-(g)_{Q_d}\right|  d x\right)^2d t\right)^{1 / 2} \\
& \leq &c d^{(2-\varepsilon)/2-s}\sum_{i=0}^{\kappa} \left(\lambda^{2i s} (\lambda^{-i}d)^{-2s}\int_{t_0-d^2}^{t_0}\fint_{B_{\lambda^{-i}d}}\left|g(x,t)-(g)_{Q_d}\right|^2 d xd t\right)^{1 / 2}  \\
&=& c d^{(2-\varepsilon)/2-s} \sum_{i=0}^\kappa\lambda^{i s} \left((\lambda^{-i}d)^{-2s}\int_{t_0-d^2}^{t_0}\fint_{B_{\lambda^{-i}d}} \left|g(x,t)-(g)_{Q_d}\right|^2 d xd t\right)^{1 / 2} \\
& \stackrel{\cref{jensen}}{\leq}& c d^{(2-\varepsilon)/2-s} \sum_{i=0}^\kappa \lambda^{i s} \sum_{m=0}^i \left(\fiint_{Q_{\lambda^{-m}d}}\left|g(x,t)-(g)_{Q_{\lambda^{-m}d}}\right|^2 d x d t\right)^{1/2}.
 \end{array}
\end{equation*}
Now, changing the order of summation and estimating using \cref{av,av2}, we get
\begin{equation}{\label{T_2}}
    \begin{array}{l}
 T_2 
\leq c d^{(2-\varepsilon)/2-s} \sum_{m=0}^\kappa \left(\fiint_{Q_{\lambda^{-m}d}}\left|g(x,t)-(g)_{Q_{\lambda^{-m}d}}\right|^2 d x d t\right)^{1/2} \sum_{i=m}^\kappa \lambda^{i s}\\
\leq c d^{(2-\varepsilon)/2-s} \sum_{m=0}^\kappa \lambda^{m s} \left(\fiint_{Q_{\lambda^{-m}d}}\left|g(x,t)-(g)_{Q_{\lambda^{-m}d}}\right|^2 d x d t\right)^{1/2}\smallskip \\
 \leq c d^{(2-\varepsilon)/2-s} \sum_{m=0}^{\kappa-1} \int_{\lambda^{-m} d}^{\lambda^{-m-1} d} \lambda^{m s} \left(\fiint_{Q_{v}}\left|g(x,t)-(g)_{Q_{v}}\right|^2 d x d t\right)^{1/2} \frac{d v}{v}\\\quad+c d^{(2-\varepsilon)/2-s} \left(\frac{d}{r}\right)^s \left(\fiint_{Q_{r}}\left|g(x,t)-(g)_{Q_{r}}\right|^2 d x d t\right)^{1/2}\\ \leq c d^{(2-\varepsilon)/2-s} \sum_{m=0}^{\kappa-1} \int_{\lambda^{-m} d}^{\lambda^{-m-1} d} \left(\frac{d}{v}\right)^{ s} \left(\fiint_{Q_{v}}\left|g(x,t)-(g)_{Q_{v}}\right|^2 d x d t\right)^{1/2} \frac{d v}{v}\\\quad +c d^{(2-\varepsilon)/2-s}\left(\frac{d}{r}\right)^s \left(\fiint_{Q_{r}}\left|g(x,t)-(g)_{Q_{r}}\right|^2 d x d t\right)^{1/2}\\
 \leq c d^{(2-\varepsilon)/2-s}\bigg[ \int_d^{r}\left(\frac{d}{v}\right)^s \left(\fiint_{Q_{v}}\left|g(x,t)-(g)_{Q_{v}}\right|^2 d x d t\right)^{1/2} \frac{d v}{v}+\left(\frac{d}{r}\right)^s \left(\fiint_{Q_{r}}\left|g(x,t)-(g)_{Q_{r}}\right|^2 d x d t\right)^{1/2}\bigg],
 \end{array}
\end{equation}
for $c \equiv c(n, s)$.
\end{description}  Merging the estimates \cref{T_1,T_2} and using \cref{snail1} we get \cref{snail}.
\end{proof}
\section{H\"older continuity} Let $r<1$. For $d \leq r / 8$, we bound
\begin{equation}{\label{H\"older1}}
\begin{array}{l}
\quad\left(\fiint_{Q_{d}}\left|u(x,t)-(u)_{Q_{d}}\right|^2 d x d t\right)^{1/2} \\\stackrel{\cref{av2}}{\leq} c\left(\fiint_{Q_d}\left|w-(w)_{Q_d}\right|^2 d xdt\right)^{1 / 2}+c\left(\frac{r}{d}\right)^{n / 2 +1}\left(\fiint_{Q_{r / 4}}|u-w|^2 d x dt\right)^{1 / 2}
\\
 \stackrel{\cref{reguw}}{\leq}c\left(\frac{d}{r}\right)\left(\fiint_{Q_{r/8}}\left|w-(w)_{Q_{r/8}}\right|^2 d xdt\right)^{1 / 2}+c\left(\frac{r}{d}\right)^{n / 2 +1}\left(\fiint_{Q_{r / 4}}|u-w|^2 d x dt\right)^{1 / 2}
\\
 \stackrel{\cref{poin}}{\leq}c\left(\frac{d}{r}\right)r\left(\fiint_{Q_{r/8}}\left|\nabla w\right|^2 d xdt\right)^{1 / 2}+c\left(\frac{r}{d}\right)^{n / 2 +1}\left(\fiint_{Q_{r / 4}}|u-w|^2 d x dt\right)^{1 / 2}  \\{\leq}c\left(\frac{d}{r}\right)r\left(\fiint_{Q_{r/4}}\left|\nabla w\right|^2 d xdt\right)^{1 / 2}+c\left(\frac{r}{d}\right)^{n / 2 +1}\left(\fiint_{Q_{r / 4}}|u-w|^2 d x dt\right)^{1 / 2}  \\{\leq}c\left(\frac{d}{r}\right)r\bigg[\left(\fiint_{Q_{r/4}}\left|\nabla(u- w)\right|^2 d xdt\right)^{1 / 2}+\left(\fiint_{Q_{r/4}}\left|\nabla u\right|^2 d xdt\right)^{1 / 2}\bigg]+c\left(\frac{r}{d}\right)^{n / 2 +1}\left(\fiint_{Q_{r / 4}}|u-w|^2 d x dt\right)^{1 / 2} \\
\stackrel{\cref{caccfinal},\cref{DIFFERENCE}}{\leq} c \left[\left(\frac{d}{r}\right)+r^{\varepsilon/2}\left(\frac{r}{d}\right)^{n/2+1}\right]\bigg[\left(r^{(2-\varepsilon)}\int_{t_0-r^2}^{t_0}\left( \int_{\mathbb{R}^n \backslash B_{r }} \frac{|u(y,t)-(u)_{Q_r}|}{\left|y-x_0\right|^{n+2s }} d y\right)^2 dt\right)^{1/2}\\ \quad+\left(\fiint_{Q_r}\left|u(x,t)-(u)_{Q_r}\right|^2 d x d t\right)^{1/2}\bigg],
\end{array}
\end{equation}
with $c \equiv c(n,s)$; we have used that $r^2<r^{2-\varepsilon}$, for $r<1$ in order to use \cref{caccfinal}. The same inequality \cref{H\"older1} holds for $r / 8 \leq d \leq r$ by \cref{av}. Taking $d \equiv \tau r$ for $\tau \in(0,1 / 8)$, we get in particular
\begin{equation}{\label{H\"older2}}
\begin{array}{rcl}
     \left(\fiint_{Q_{\tau r}}\left|u(x,t)-(u)_{Q_{\tau r}}\right|^2 d x d t\right)^{1/2} &\leq &c\left(\tau+r^{ \varepsilon/ 2} \tau^{-n / 2-1}\right) \bigg[\left(\fiint_{Q_r}\left|u(x,t)-(u)_{Q_r}\right|^2 d x d t\right)^{1/2}\\ &&+\left(r^{(2-\varepsilon)}\int_{t_0-r^2}^{t_0}\left( \int_{\mathbb{R}^n \backslash B_{r }} \frac{|u(y,t)-(u)_{Q_r}|}{\left|y-x_0\right|^{n+2s }} d y\right)^2 dt\right)^{1/2}\bigg],
\end{array}
\end{equation}
with $c \equiv c(n,s)$. In order to get a full decay estimate, we need to evaluate the nonlocal term. For this we use \cref{snail}, that yields
\begin{equation}{\label{H\"older3}}
    \begin{array}{l}
       \quad\left((\tau r)^{(2-\varepsilon)}\int_{t_0-(\tau r)^2}^{t_0}\left( \int_{\mathbb{R}^n \backslash B_{\tau r }} \frac{|u(y,t)-(u)_{Q_{\tau r}}|}{\left|y-x_0\right|^{n+2s }} d y\right)^2 dt\right) \\ \leq c \underbrace{\tau^{(2-\varepsilon)}\left(r^{(2-\varepsilon)}\int_{t_0-r^2}^{t_0}\left( \int_{\mathbb{R}^n \backslash B_{r }} \frac{|u(y,t)-(u)_{Q_r}|}{\left|y-x_0\right|^{n+2s }} d y\right)^2 dt\right)}_{S_1}\\\quad+c\underbrace{(\tau r)^{(2-\varepsilon)}\left(\int_{\tau r}^r\left(\fiint_{Q_v}\left|u(x,t)-(u)_{Q_v}\right|^2 d x d t\right)^{1/2}\frac{dv}{v^{1+s}}\right)^2}_{S_2}+c\underbrace{(\tau r)^{(2-\varepsilon)}r^{-2s}\left(\fiint_{Q_r}\left|u(x,t)-(u)_{Q_r}\right|^2 d x d t\right)}_{S_3}.
    \end{array}
\end{equation}
\begin{description}[leftmargin=0cm]
    \item{\textbf{Estimate for $S_1$:}} Clearly 
    \begin{equation}{\label{snail1st}}
        \begin{array}{l}
             S_1\leq \tau ^{(2-\varepsilon)}\bigg[\left(r^{(2-\varepsilon)}\int_{t_0-r^2}^{t_0}\left( \int_{\mathbb{R}^n \backslash B_{r }} \frac{|u(y,t)-(u)_{Q_r}|}{\left|y-x_0\right|^{n+2s }} d y\right)^2 dt\right)+ \fiint_{Q_r}\left|u(x,t)-(u)_{Q_r}\right|^2 d x d t\bigg].
        \end{array}
    \end{equation}
    \item{\textbf{Estimate for $S_2$:}} Note that $r<1$ and $\varepsilon <1-s$ implies $\varepsilon<2-2s$. By \cref{H\"older1} (this holds for $r/8\leq d\leq r$ by \cref{av}), and Young's inequality, we get
    \begin{equation}{\label{snail2nd}}
        \begin{array}{l}
              S_2 \leq c \Bigg[\tau^{2-\varepsilon}r^{2-\varepsilon-2}\left(\int_{\tau r}^{r} \frac{d v}{v^{1+s-1}}\right)^2 
 + \tau^{2-\varepsilon}r^{2-\varepsilon+\varepsilon+n+ 2 }\left(\int_{\tau r}^{r} \frac{d v}{v^{2+s+n/2}}\right)^2\Bigg]\Bigg[ \fiint_{Q_r}\left|u(x,t)-(u)_{Q_r}\right|^2 d x d t\\\qquad\qquad\qquad\qquad\qquad\qquad\qquad\qquad\qquad\qquad+\left(r^{(2-\varepsilon)}\int_{t_0-r^2}^{t_0}\left( \int_{\mathbb{R}^n \backslash B_{r }} \frac{|u(y,t)-(u)_{Q_r}|}{\left|y-x_0\right|^{n+2s }} d y\right)^2 dt\right)\quad\Bigg]\\\leq c \Bigg[ \tau^{2-\varepsilon}r^{2-\varepsilon-2s}+ \tau^{-2s-n -2} r^{2-2s}\Bigg]\Bigg[\left(r^{(2-\varepsilon)}\int_{t_0-r^2}^{t_0}\left( \int_{\mathbb{R}^n \backslash B_{r }} \frac{|u(y,t)-(u)_{Q_r}|}{\left|y-x_0\right|^{n+2s }} d y\right)^2 dt\right)\\\qquad\qquad\qquad\qquad\qquad\qquad\qquad\qquad
 + \fiint_{Q_r}\left|u(x,t)-(u)_{Q_r}\right|^2 d x d t\quad\Bigg]\\\leq c \Bigg[ \tau^{2-\varepsilon}+ \tau^{-2s-n -2} r^{\varepsilon}\Bigg]\Bigg[\left(r^{(2-\varepsilon)}\int_{t_0-r^2}^{t_0}\left( \int_{\mathbb{R}^n \backslash B_{r }} \frac{|u(y,t)-(u)_{Q_r}|}{\left|y-x_0\right|^{n+2s }} d y\right)^2 dt\right)+ \fiint_{Q_r}\left|u(x,t)-(u)_{Q_r}\right|^2 d x d t\quad\Bigg],
\end{array}
\end{equation}
where $c \equiv c( n,s)$.
\item{\textbf{Estimate for $S_3$:}} As the same argument of $S_2$, it follows that
\begin{equation}{\label{snail3rd}}
    \begin{array}{l}
             S_3\leq \tau ^{(2-\varepsilon)}\bigg[\left(r^{(2-\varepsilon)}\int_{t_0-r^2}^{t_0}\left( \int_{\mathbb{R}^n \backslash B_{r }} \frac{|u(y,t)-(u)_{Q_r}|}{\left|y-x_0\right|^{n+2s }} d y\right)^2 dt\right)+ \fiint_{Q_r}\left|u(x,t)-(u)_{Q_r}\right|^2 d x d t\bigg].
        \end{array}
    \end{equation}
\end{description}
Combining \cref{snail1st,snail2nd,snail3rd}, by \cref{H\"older3} we get
\begin{equation}{\label{H\"older4}}
    \begin{array}{l}
        \quad (\tau r)^{(2-\varepsilon)}\int_{t_0-(\tau r)^2}^{t_0}\left( \int_{\mathbb{R}^n \backslash B_{\tau r }} \frac{|u(y,t)-(u)_{Q_{\tau r}}|}{\left|y-x_0\right|^{n+2s }} d y\right)^2 dt\\\leq c[\tau^{2-\varepsilon}+r^{\varepsilon}\tau^{-2s-n-2}]\times\bigg[r^{(2-\varepsilon)}\int_{t_0-r^2}^{t_0}\left( \int_{\mathbb{R}^n \backslash B_{r }} \frac{|u(y,t)-(u)_{Q_r}|}{\left|y-x_0\right|^{n+2s }} d y\right)^2 dt \smallskip+ \fiint_{Q_r}\left|u(x,t)-(u)_{Q_r}\right|^2 d x d t\bigg].
    \end{array}
\end{equation}
Merging \cref{H\"older2,H\"older4}, we thus get
\begin{equation}{\label{H\"oldermain}}
\begin{array}{l}
       \Bigg[\left((\tau r)^{(2-\varepsilon)}\int_{t_0-(\tau r)^2}^{t_0}\left( \int_{\mathbb{R}^n \backslash B_{\tau r }} \frac{|u(y,t)-(u)_{Q_{\tau r}}|}{\left|y-x_0\right|^{n+2s }} d y\right)^2 dt\right)^{1/2}+ \left(\fiint_{Q_{\tau r}}\left|u(x,t)-(u)_{Q_{\tau r}}\right|^2 d x d t\right)^{1/2}\Bigg]\\\leq c[\tau^{(2-\varepsilon)/2}+r^{ \varepsilon/ 2} \tau^{-s-n / 2-1}] \bigg[\left(r^{(2-\varepsilon)}\int_{t_0-r^2}^{t_0}\left( \int_{\mathbb{R}^n \backslash B_{r }} \frac{|u(y,t)-(u)_{Q_r}|}{\left|y-x_0\right|^{n+2s }} d y\right)^2 dt\right)^{1/2}\\\quad\qquad\qquad\quad\qquad\qquad\quad\qquad\qquad+\left(\fiint_{Q_r}\left|u(x,t)-(u)_{Q_r}\right|^2 d x d t\right)^{1/2}\bigg].
\end{array}
\end{equation}
Now fix $\alpha_0 \in (0,1)$, choose $\alpha$ such that $\alpha_0<\alpha<1
$
and take $\alpha_1=(1+\alpha)/2$.
We eventually determine $\tau \equiv \tau(n,s, \alpha) \leq 1 / 8$ such that
\begin{equation}{\label{tau}}
\begin{array}{l}
2 c \tau^{(2-\varepsilon)/ 2}\tau^{-(2-\varepsilon)\alpha_1/2}  \leq 1 \qquad\text{ and }\qquad
\tau^{(2-\varepsilon)(1-\alpha) / 4} \leq \frac{1}{2}.
\end{array}    
\end{equation}
As $\tau$ is fixed, we find $r_* \equiv r_*(n,s ,\alpha)\in(0,1)$ such that if $r \leq \varrho \leq r_*$, then $2 c r^{ \varepsilon/ 2} \tau^{-s-n / 2-1 -(2-\varepsilon)\alpha_1/2} \leq 1$. So \cref{H\"oldermain} gives
\begin{equation}{\label{H\"oldermain2}}
\begin{array}{l}
   \quad    \Bigg[\left((\tau r)^{(2-\varepsilon)}\int_{t_0-(\tau r)^2}^{t_0}\left( \int_{\mathbb{R}^n \backslash B_{\tau r }} \frac{|u(y,t)-(u)_{Q_{\tau r}}|}{\left|y-x_0\right|^{n+2s }} d y\right)^2 dt\right)^{1/2}+ \left(\fiint_{Q_{\tau r}}\left|u(x,t)-(u)_{Q_{\tau r}}\right|^2 d x d t\right)^{1/2}\Bigg]\\\leq \tau^{\frac{(2-\varepsilon)\alpha_1}{2}} \bigg[\left(r^{(2-\varepsilon)}\int_{t_0-r^2}^{t_0}\left( \int_{\mathbb{R}^n \backslash B_{r }} \frac{|u(y,t)-(u)_{Q_r}|}{\left|y-x_0\right|^{n+2s }} d y\right)^2 dt\right)^{1/2}+\left(\fiint_{Q_r}\left|u(x,t)-(u)_{Q_r}\right|^2 d x d t\right)^{1/2}\bigg],
\end{array}
\end{equation}
which holds whenever $r\leq\varrho  \leq r_*<1$. Now we introduce the sharp fractional maximal type operator
\begin{equation}{\label{maxi}}
\begin{array}{rcl}
\mathrm{M}(x_0,r)&:=&\sup _{ v \leq r} v^{-\frac{(2-\varepsilon)\alpha}{2}} \bigg[\left(v^{(2-\varepsilon)}\int_{t_0-v^2}^{t_0}\left( \int_{\mathbb{R}^n \backslash B_{v }} \frac{|u(y,t)-(u)_{Q_v}|}{\left|y-x_0\right|^{n+2s }} d y\right)^2 dt\right)^{1/2}\\&&\qquad\qquad\qquad\qquad+\left(\fiint_{Q_v}\left|u(x,t)-(u)_{Q_v}\right|^2 d x d t\right)^{1/2}\bigg]
\end{array}
\end{equation}
and its truncated version for $0<\delta<\frac{1}{2}$,
\begin{equation}{\label{maximal}}
\begin{array}{rcl}
\mathrm{M}_{\delta}(x_0, r)&:=&\sup _{\delta r \leq v \leq r} v^{-\frac{(2-\varepsilon)\alpha}{2}} \bigg[\left(v^{(2-\varepsilon)}\int_{t_0-v^2}^{t_0}\left( \int_{\mathbb{R}^n \backslash B_{v }} \frac{|u(y,t)-(u)_{Q_v}|}{\left|y-x_0\right|^{n+2s }} d y\right)^2 dt\right)^{1/2}\\&&\qquad\qquad\qquad\qquad+\left(\fiint_{Q_v}\left|u(x,t)-(u)_{Q_v}\right|^2 d x d t\right)^{1/2}\bigg].
\end{array}
\end{equation}
Multiplying both sides of \cref{H\"oldermain2} by $(\tau r)^{-(2-\varepsilon)\alpha/2}$ and taking the sup with respect to $r \in(\delta\varrho, \varrho)$, one gets
\begin{equation*}
\begin{array}{rcl}
\mathrm{M}_{\delta}(x_0, \tau \varrho) & \leq &\tau^{(2-\varepsilon)(1-\alpha) / 4} \sup _{\delta \varrho \leq v \leq \varrho}  v^{-\frac{(2-\varepsilon)\alpha}{2}} \bigg[\left(v^{(2-\varepsilon)}\int_{t_0-v^2}^{t_0}\left( \int_{\mathbb{R}^n \backslash B_{v }} \frac{|u(y,t)-(u)_{Q_v}|}{\left|y-x_0\right|^{n+2s }} d y\right)^2 dt\right)^{1/2}\\&&\quad\qquad\qquad\quad\qquad\qquad\quad\qquad\quad+\left(\fiint_{Q_v}\left|u(x,t)-(u)_{Q_v}\right|^2 d x d t\right)^{1/2}\bigg] \\
& \leq& \tau^{(2-\varepsilon)(1-\alpha) / 4} \sup _{\delta \tau \varrho \leq v \leq \varrho}  v^{-\frac{(2-\varepsilon)\alpha}{2}} \bigg[\left(v^{(2-\varepsilon)}\int_{t_0-v^2}^{t_0}\left( \int_{\mathbb{R}^n \backslash B_{v }} \frac{|u(y,t)-(u)_{Q_v}|}{\left|y-x_0\right|^{n+2s }} d y\right)^2 dt\right)^{1/2}\\&&\quad\quad\qquad\qquad\quad\qquad\qquad\quad\qquad+\left(\fiint_{Q_v}\left|u(x,t)-(u)_{Q_v}\right|^2 d x d t\right)^{1/2}\bigg] \\
& \stackrel{\cref{tau}}{\leq} & \frac{1}{2} \mathrm{M}_{\delta}(x_0, \tau \varrho
)+\sup _{\tau \varrho \leq v \leq \varrho}  v^{-\frac{(2-\varepsilon)\alpha}{2}} \bigg[\left(v^{(2-\varepsilon)}\int_{t_0-v^2}^{t_0}\left( \int_{\mathbb{R}^n \backslash B_{v }} \frac{|u(y,t)-(u)_{Q_v}|}{\left|y-x_0\right|^{n+2s }} d y\right)^2 dt\right)^{1/2}\\&&\quad\qquad\qquad\quad\qquad\qquad\quad\qquad\qquad+\left(\fiint_{Q_v}\left|u(x,t)-(u)_{Q_v}\right|^2 d x d t\right)^{1/2}\bigg].
\end{array}
\end{equation*}
The above along with the fact that $\mathrm{M}_{\delta}$ is always finite and $\tau \equiv \tau(n,s, \alpha)$, implies
\begin{equation*}
\begin{array}{rcl}
\mathrm{M}_{\delta}(x_0, \varrho)& \leq& \frac{c}{\varrho^{\frac{(2-\varepsilon)\alpha}{2}} }\sup _{\tau \varrho \leq v \leq \varrho} \bigg[\left(v^{(2-\varepsilon)}\int_{t_0-v^2}^{t_0}\left( \int_{\mathbb{R}^n \backslash B_{v }} \frac{|u(y,t)-(u)_{Q_v}|}{\left|y-x_0\right|^{n+2s }} d y\right)^2 dt\right)^{1/2}\\&&\quad\qquad\qquad\quad\qquad\qquad\quad+\left(\fiint_{Q_v}\left|u(x,t)-(u)_{Q_v}\right|^2 d x d t\right)^{1/2}\bigg],
\end{array}
\end{equation*}
with $c\equiv c(n,s,\alpha)$.
Thus $\delta \rightarrow 0$ yields
\begin{equation*}
\begin{array}{rcl}
\mathrm{M}(x_0, \varrho)& \leq& \frac{c}{\varrho^{\frac{(2-\varepsilon)\alpha}{2}} }\sup _{\tau \varrho \leq v \leq \varrho} \bigg[\left(v^{(2-\varepsilon)}\int_{t_0-v^2}^{t_0}\left( \int_{\mathbb{R}^n \backslash B_{v }} \frac{|u(y,t)-(u)_{Q_v}|}{\left|y-x_0\right|^{n+2s }} d y\right)^2 dt\right)^{1/2}\\&&\quad\qquad\qquad\quad\qquad\qquad\quad+\left(\fiint_{Q_v}\left|u(x,t)-(u)_{Q_v}\right|^2 d x d t\right)^{1/2}\bigg],
\end{array}
\end{equation*}
with $c\equiv c(n,s ,\alpha)$. In order to estimate the right hand side of the last inequality, we use \cref{av,snail} to get, whenever $0< r \leq \varrho \leq r_*$,
\begin{equation*}
\begin{array}{rcl}
\mathrm{M}(x_0, \varrho)& \leq &\frac{c}{\varrho^{\frac{(2-\varepsilon)\alpha}{2}} } \bigg[\left(\varrho^{(2-\varepsilon)}\int_{t_0-\varrho^2}^{t_0}\left(\int_{\mathbb{R}^n \backslash B_{\varrho }} \frac{|u(y,t)-(u)_{Q_\varrho}|}{\left|y-x_0\right|^{n+2s }} d y\right)^2 dt\right)^{1/2}\\&&\qquad\qquad\qquad\qquad\qquad+\left(\fiint_{Q_\varrho}\left|u(x,t)-(u)_{Q_\varrho}\right|^2 d x d t\right)^{1/2}\bigg].
\end{array}
\end{equation*}
Now by \cref{maxi},
\begin{equation}{\label{H\"olderfinal}}
\begin{array}{l}
   \quad  \bigg[\left(r^{(2-\varepsilon)}\int_{t_0-r^2}^{t_0}\left( \int_{\mathbb{R}^n \backslash B_{r }} \frac{|u(y,t)-(u)_{Q_r}|}{\left|y-x_0\right|^{n+2s }} d y\right)^2 dt\right)^{1/2}+\left(\fiint_{Q_r}\left|u(x,t)-(u)_{Q_r}\right|^2 d x d t\right)^{1/2}\bigg]\\\leq c\left(\frac{r}{\varrho}\right)^{\frac{(2-\varepsilon)\alpha}{2}}\bigg[\left(\varrho^{(2-\varepsilon)}\int_{t_0-\varrho^2}^{t_0}\left(\int_{\mathbb{R}^n \backslash B_{\varrho }} \frac{|u(y,t)-(u)_{Q_\varrho}|}{\left|y-x_0\right|^{n+2s }} d y \right)^2dt\right)^{1/2}+\left(\fiint_{Q_\varrho}\left|u(x,t)-(u)_{Q_\varrho}\right|^2 d x d t\right)^{1/2}\bigg],
\end{array}
\end{equation}
for every $0< r \leq \varrho \leq r_*$, where $c\equiv 
 c(n,s, \alpha)$. 
 Now in order to recover $\alpha_0$, we first choose $0<\varepsilon<\operatorname{min}\{1-\alpha_0,1-s\}$. Now we choose $\alpha=\frac{2\alpha_0}{(2-\varepsilon)}$. Clearly $\alpha_0<\alpha<1.$ With such choices, \cref{H\"olderfinal} becomes 
\begin{equation*}
\begin{array}{l}
  \quad   \bigg[\left(r^{(2-\varepsilon)}\int_{t_0-r^2}^{t_0}\left( \int_{\mathbb{R}^n \backslash B_{r }} \frac{|u(y,t)-(u)_{Q_r}|}{\left|y-x_0\right|^{n+2s }} d y\right)^2 dt\right)^{1/2}+\left(\fiint_{Q_r}\left|u(x,t)-(u)_{Q_r}\right|^2 d x d t\right)^{1/2}\bigg]\\\leq c\left(\frac{r}{r_*}\right)^{\alpha_0}\bigg[\left(r_*^{(2-\varepsilon)}\int_{t_0-r_*^2}^{t_0}\left(\int_{\mathbb{R}^n \backslash B_{r_* }} \frac{|u(y,t)-(u)_{Q_{r_*}}|}{\left|y-x_0\right|^{n+2s }} d y\right)^2 dt\right)^{1/2}\\\quad+\left(\fiint_{Q_{r_*}}\left|u(x,t)-(u)_{Q_{r_*}}\right|^2 d x d t\right)^{1/2}\bigg],
     \end{array}
     \end{equation*}
     which implies
     \begin{equation}{\label{H\"olderfinal2}}
     \begin{array}{l}
    \fiint_{Q_r}\left|u(x,t)-(u)_{Q_r}\right|^2 d x d t\leq c\left(\frac{r}{r_*}\right)^{2\alpha_0}\bigg[r_*^{(2-\varepsilon)}\int_{t_0-r_*^2}^{t_0}\left(\int_{\mathbb{R}^n \backslash B_{r_* }} \frac{|u(y,t)-(u)_{Q_{r_*}}|}{\left|y-x_0\right|^{n+2s }} d y\right)^2 dt\\\qquad\qquad\qquad\qquad\qquad\qquad\qquad\qquad\qquad\qquad\quad+\fiint_{Q_{r_*}}\left|u(x,t)-(u)_{Q_{r_*}}\right|^2 d x d t\bigg],
\end{array}
\end{equation}
where $c\equiv 
 c(n,s, \alpha_0)$. Further estimating using \cref{av2} 
 we get that
\begin{equation}{\label{tail2}}
\begin{array}{l}
 \quad\bigg[r_*^{2-\varepsilon}\int_{t_0-r_*^2}^{t_0}\left(\int_{\mathbb{R}^n \backslash B_{r_* }} \frac{|u(y,t)-(u)_{Q_{r_*}}|}{\left|y-x_0\right|^{n+2s }} d y\right)^2 dt+\fiint_{Q_{r_*}}\left|u(x,t)-(u)_{Q_{r_*}}\right|^2 d x d t\bigg]\\\leq c\Bigg[2r_*^{2-\varepsilon}r_*^{2}\fint_{t_0-r_*^2}^{t_0}\left(\int_{\mathbb{R}^n \backslash B_{r_* }} \frac{|u(y,t)|}{\left|y-x_0\right|^{n+2s }} d y\right)^2 dt+2r_*^{2-2s-\varepsilon}r_*^{2-2s}\left(\fiint_{Q_{r_*}}\left|u(x,t)\right| d x d t\right)^2\\\quad+\fiint_{Q_{r_*}}\left|u(x,t)\right|^2 d x d t\Bigg]\\\leq
 \bigg[
 r_*^{2-2s-\varepsilon}r_*^{2-2s}\underset{t_0-r_*^2<t<t_0 }{\operatorname{ess} \sup}\left( r_*^{2s}\int_{\mathbb{R}^n \backslash B_{r_*}} \frac{|u(y, t)|}{\left|y-x_0\right|^{n+2s}} d y\right)^2+\fiint_{Q_{r_*}}\left|u(x,t)\right|^2 d x d t\bigg]\\
\leq \frac{c}{r_*^{n+2}}\bigg[
\left(\operatorname{Tail}_{\infty}(u;x_0,r_*,t_0-r_*^2,t_0)\right)^2+||u||^2_{L^2\left(0,T;L^2(\Omega)\right)}\bigg],
\end{array}
 \end{equation}
where $c\equiv 
 c(n,s)$. In order to get the last two lines, we have used the fact that $\varepsilon<1-s<2-2s<2$ and $r_*\in(0,1)$. Now by \cref{H\"olderfinal2} and \cref{tail2} we get
 \begin{equation}{\label{H\"olderfinal3}}
\begin{array}{l}
\fiint_{Q_r}\left|u(x,t)-(u)_{Q_r}\right|^2 d x d t\leq c\left(\frac{r}{r_*}\right)^{2\alpha_0} \frac{1}{r_*^{n+2}}\bigg[\left(\operatorname{Tail}_{\infty}(u;x_0,r_*,t_0-r_*^2,t_0)\right)^2+||u||^2_{L^2\left(0,T;L^2(\Omega)\right)}
\bigg],
\end{array}
\end{equation}
where $c\equiv 
 c(n,s, \alpha_0)$.
 Since $\operatorname{Tail}_{\infty}(u ; x_0, r, t_0-r^2,t_0)$, is finite for each $B_r\Subset\Omega$ and $(t_0-r^2,t_0)\Subset(0,T)$, for our solution $u$, so the right hand side of \cref{H\"olderfinal3} is finite. Therefore we get by \cref{H\"olderfinal3} that for every $\alpha_0\in(0,1)$, there exists $r_*\equiv r_*(n,s,\alpha_0)\in (0,1)$ and $c \equiv c(n,s,\alpha_0,||u||_{L^2\left(0,T;L^2(\Omega)\right)},||u||_{\operatorname{Tail}_{\infty}(u;x_0,t_0,r_*)}
 )\equiv c(\operatorname{data}_h(\alpha_0),\alpha_0)\geq 1 $ such that
\begin{equation}{\label{continuity}}
    \fint_{t_0-r^2}^{t_0} \fint_{B_r\left(x_0\right)}\left|u(x,t)-(u)_{B_r\times(t_0-r^2,t_0)}\right|^2 d x dt\leq c\left(\frac{r}{r_*}\right)^{2\alpha_0}
\end{equation}
holds whenever $B_r\Subset \Omega$, $(t_0-r^2,t_0)\Subset (0,T)$ and $r\leq r^*$.
This gives us the H\"older continuity by the Campanato-Meyers integral characterization, for all $\alpha_0\in(0,1)$. For details, one can see \cite{giusti2003direct}. So $u$ is locally Hölder continuous for all exponent in $(0,1).$
Now using the H\"older continuity we prove that $\nabla u$ is also in $C^{\alpha}$ for some $\alpha\in(0,1)$.
\section{Gradient H\"older regularity}
Throughout this section, we consider arbitrary open subsets $\Omega_0 \Subset \Omega_1 \Subset \Omega$ and open intervals $(\tau_1,\tau_2)\Subset (t_1,t_2)\Subset(0,T)$, and denote $d:=\min \left\{\operatorname{dist}\left(\Omega_0, \Omega_1\right),
\operatorname{dist}\left(\Omega_1, \Omega\right),\operatorname{dist}\left((t_1,t_2), (0,T)\right),\operatorname{dist}\left((\tau_1,\tau_2), (t_1,t_2)\right), 1\right\}$ (as the distance between two subsets of a topological space is a nonnegative number and here we are comparing those distances, $d$ is well defined). We take $B_{r} \equiv B_{r}(x_0) \Subset \Omega_1$ with $x_0 \in \Omega_0$; $(t_0-r^2,t_0)\Subset(t_1,t_2)$ with $t_0\in (\tau_1,\tau_2)$; and $0<r \leq d / 4$ and all the balls will be centred at $x_0$. Moreover, $\beta, \lambda$ will be real numbers satisfying $s<\beta<1$ and $\lambda>0$, and their precise values will be prescribed later depending on the context. We will use H\"older continuity in the form $\|u\|_{C^{0, \beta}\left(\Omega_1\times (t_1,t_2)\right)} \leq c \equiv c(\operatorname{data}_h(\beta), \beta, d)$, for each $\beta<1$.
\begin{lemma} 
If $s<\beta<1$, then
\begin{equation}{\label{decay1}}
\fint_{t_0-r^2/4}^{t_0}\int_{B_{r/ 2}} \fint_{B_{r / 2}} \frac{|u(x,t)-u(y,t)|^2}{|x-y|^{n+2s}} d x  d yd t \leq c r^{2(\beta-s)}    
\end{equation}
holds with $c \equiv c(\operatorname{data}_h(\beta), d, \beta)$.\\
The inequality
\begin{equation}{\label{decay2}}
\int_{t_0-v^2}^{t_0}\left(\int_{\mathbb{R}^n \backslash B_{v}} \frac{|u(y,t)-(u)_{Q_{v}}|}{\left|y-x_0\right|^{n+2s }} d y\right)^2 dt 
\leq c  v^{2-2s}  
\end{equation}
holds whenever $0<v \leq r$, where $c \equiv c(\operatorname{data}_h, d)$.\\
If $\lambda>0$ then
\begin{equation}{\label{decay3}}
\fint_{t_0-r^2/4}^{t_0}\fint_{B_{r / 2}}|\nabla u|^2 d x dt\leq c r^{-2 \lambda}
\end{equation}
holds with $c \equiv c(\operatorname{data}_h(\lambda), d, \lambda)$.
\end{lemma}
\begin{proof}
    Estimate in \cref{decay1} follows from H\"older continuity with $\alpha_0 \equiv \beta$. To prove \cref{decay2} we notice from the proof of \cref{decay} (the part estimate of $T_1$), that 
    \begin{equation}{\label{oneresult}}
    \begin{array}{rcl}
        |(u)_{Q_{d}}-(u)_{Q_{v}}|^2&\leq& c\Bigg[\int_v^d \left(\fiint_{Q_\varrho}\left|u(x,t)-(u)_{Q_\varrho}\right|^2 d x d t \right)^{1/2}\frac{d \varrho}{\varrho} 
+ \left(\fiint_{Q_d}\left|u(x,t)-(u)_{Q_d}\right|^2 d x d t\right)^{1/2}\Bigg]^2\\&\leq&c\left(\int_v^d \varrho^{\beta-1}d\varrho+d^{\beta}\right)^2\leq c d^{2\beta}, \end{array}
    \end{equation}
    where $c \equiv c(\operatorname{data}_h(\beta), d, \beta)$.
    Now using 
    \cref{oneresult} we estimate as follows:
\begin{equation*}
\begin{array}{l}
\quad\int_{t_0-v^2}^{t_0}\left(\int_{\mathbb{R}^n \backslash B_{v}} \frac{\left|u(y,t)-(u)_{Q_{v}}\right|}{\left|y-x_0\right|^{n+2s }} d y\right)^2 dt 
\\\leq cv^2\fint_{t_0-v^2}^{t_0}\left(\int_{\mathbb{R}^n \backslash B_{d}} \frac{\left|u(y,t)-(u)_{Q_d}\right|}{\left|y-x_0\right|^{n+2s}} d y\right)^2dt +c v^2\fint_{t_0-v^2}^{t_0}\left(\int_{\mathbb{R}^n \backslash B_{d}} \frac{|(u)_{Q_{d}}-(u)_{Q_{v}}|}{\left|y-x_0\right|^{n+2s }} d y\right)^2 dt  \\\quad+cv^{2}\fint_{t_0-v^2}^{t_0}\left(\int_{B_d \backslash B_v} \frac{\left|u(y,t)-u(x_0,t)\right|}{\left|y-x_0\right|^{n+2s }} d y \right)^2dt+cv^{2}\fint_{t_0-v^2}^{t_0}\left(\int_{B_{d }\backslash B_v} \frac{\left|u(x_0,t)-(u)_{Q_v}\right|}{\left|y-x_0\right|^{n+2s}} d y\right)^2dt\\
\leq  cv^{2-2s}\left(\frac{v}{d}\right)^{2s}\left(d^{-2s}\underset{t_0-d^2<t<t_0}{\operatorname{ess}\text{sup}}\left(d^{2s}\int_{\mathbb{R}^n\backslash{B_d} } \frac
{\left|u(y,t)\right|}{\left|y-x_0\right|^{n+2s}} d y\right)^2
+\frac{1}{d^{n+2+2s}}\|u\|_{L^2\left(\Omega_1\times(t_1,t_2)\right)}^2\right)\\\quad+cv^{2-2s}\left(\frac{v}{d}\right)^{2s} d^{-2s}|(u)_{Q_{d}}-(u)_{Q_{v}}|^2 +c v^{2}v^{-2s}\left(\int_{B_{d}}\frac{d y}{\left|y-x_0\right|^{n+(s-\beta) }}[u]_{C^{0, \beta} ( \Omega_1\times(t_1,t_2))}\right)^2 \\\quad+c v^{2}v^{2\beta } \left(\int_{\mathbb{R}^n \backslash B_v} \frac{d y}{\left|y-x_0\right|^{n+2s}}[u]_{C^{0, \beta }( \Omega_1\times(t_1,t_2))}\right)^2 \\
\leq  \frac{cv^{2-2s}}{d^{n+2+2s}}\left(
{\left(\operatorname{Tail}_\infty(u;x_0,t_0,d)\right)^2}
+\|u\|_{L^2\left(\Omega_1\times(t_1,t_2)\right)}^2\right)\\\quad+cv^{2-2s} d^{-2s}\left|(u)_{Q_{d}}-(u)_{Q_{v}}\right|^2
+c v^{2-2s}d^{2\beta-2s}  +c v^{2-2s}v^{2\beta -2s} 
\\\leq  v^{2-2s}[c d^{-n-2-2s }+c d^{2(\beta-s)}+c v^{2(\beta-s) } ]\leq cv^{2-2s},
\end{array}
\end{equation*}
where $c \equiv c(\operatorname{data}_h(\beta), d, \beta)$, this gives \cref{decay2} if we choose $\beta=(1+s)/2$.\\
To prove \cref{decay3} we estimate using \cref{caccfinal} and \cref{continuity} as
\begin{equation*}
r^{-2}\fint_{t_0-r^2}^{t_0} \fint_{B_r}\left|u(x,t)-(u)_{B_r\times(t_0-r^2,t_0)}\right|^2 d x dt+r^{-2s}\fint_{t_0-r^2}^{t_0} \fint_{B_r}\left|u(x,t)-(u)_{B_r\times(t_0-r^2,t_0)}\right|^2 d x dt\leq c r^{2
(\beta-1)}
\end{equation*}
holds with $c \equiv c(\operatorname{data}_h(\beta), \beta)$. By \cref{decay2} we instead have
\begin{equation*}
\int_{t_0-r^2}^{t_0}\left(\int_{\mathbb{R}^n \backslash B_{r}} \frac{|u(y,t)-(u)_{Q_{r}}|}{\left|y-x_0\right|^{n+2s }} d y\right)^2 dt\leq c \leq c r^{2(\beta-1)} 
\end{equation*}
holds with $c \equiv c(\operatorname{data}_h, d)$. Taking $\beta$ such that $1-\beta \leq \lambda$, we get \cref{decay3}.
\end{proof}
\begin{lemma}
    Let $w$ be as defined in \cref{heat}, then
\begin{equation}{\label{decay4}}
\fint_{t_0-r^2/16}^{t_0}\fint_{B_{r/ 4}\left(x_0\right)}|\nabla u-\nabla w|^2 d x dt\leq c r^{2(1-s)}  
\end{equation}
holds where $c \equiv c(\operatorname{data}_h, d)$.
\end{lemma}
\begin{proof}
The proof follows directly from \cref{diffinal}, by choosing $k=(u)_{Q_r}$ and using \cref{decay1}, \cref{decay2} along with Hölder continuity for the exponent $\beta=(1+s)/2$. For the sake of completeness, we include here some details. Let us recall the proof of \cref{diffest}, and following the notations introduced, we have
\begin{equation}
	\label{differ+}
	\begin{array}{l}
		 \fint_{t_0-r^2/16}^{t_0} \fint_{B_{r / 4}}\left|\nabla u-\nabla w\right|^2 dx dt
   =\underbrace{-\fint_{t_0-r^2/16}^{t_0} \fint_{B_{r /4}}u_{t} (u-w)dx dt +\fint_{t_0-r^2/16}^{t_0} \fint_{B_{r / 4}}w_{t} (u-w) dx dt}_{(I)}
\\\qquad\qquad\qquad\qquad\qquad\qquad\qquad
		  \underbrace{- \frac{c}{2}\fint_{t_0-r^2/16}^{t_0}\int_{B_{r/2}} \fint_{B_{r/2 }}(u(x,t)-u(y,t))((u-w)(x,t)-(u-w)(y,t)) d\mu dt}_{(II)} \\\qquad\qquad\qquad\qquad\qquad\qquad\qquad
		 \underbrace{ -c \fint_{t_0-r^2/16}^{t_0}\int_{\mathbb{R}^n \backslash B_{r/2 }}  \fint_{B_{r /2}}(u(x,t)-u(y,t)) (u-w)(x,t) \times K(x, y, t) d x  dy dt}_{(III)}.
	\end{array}
\end{equation} We have ignored the Steklov averages in \cref{differ+}, as their convergence properties will give us the same results we got in \cref{diffest}. We now improve the estimates for the terms $\mathrm{(II)}$ and $\mathrm{(III)}$ compared to \cref{diffest}.
\begin{description}[leftmargin=0cm]
        \item{\textbf{Estimate for (II):}} Using \cref{decay1} in \cref{diff2}, we get choosing $\varepsilon=\frac{1}{4}$,
        \begin{equation}{\label{II}}
        \begin{array}{rcl}
           {|(\mathrm{II})|}  \leq Cr^{2(1-s)}r^{2(\beta-s)}+\frac{1}{4} \fint_{t_0-r^2/16}^{t_0} \fint_{B_{r/4}} \left|\nabla(u-w)\right|^2 d x d t,
            \end{array}
        \end{equation}
whenever $s<\beta<1$, where $c \equiv c(\operatorname{data}_h(\beta), d, \beta)$.
\item{\textbf{Estimate for (III):}}  From \cref{diff3}, taking $\varepsilon=\frac{1}{4}$, we get using \cref{continuity,decay2},
\begin{equation}{\label{III}}
	\begin{array}{rcl}
    {|(\mathrm{III})|}&\leq& Cr^{2(1-s) }\Bigg[r^{-2s}\fint_{t_0-r^2}^{t_0} \fint_{B_{r }}|u(x,t)-(u)_{Q_r}|^2dxdt+r^{-2+2s}\int_{t_0-r^2}^{t_0}\left( \int_{\mathbb{R}^n \backslash B_{r }} \frac{|u(y,t)-(u)_{Q_r}|}{\left|y-x_0\right|^{n+2s }} d y\right)^2 dt\Bigg]\\&&+\frac{1}{4}\fint_{t_0-r^2/16}^{t_0}\fint_{B_{r/4 }}|\nabla(u-w)(x,t)|^2 dxdt\\&\leq& Cr^{2(1-s) }[r^{2(\beta-s)}+r^{-2+2s}r^{2-2s}]+\frac{1}{4}\fint_{t_0-r^2/16}^{t_0}\fint_{B_{r/4 }}|\nabla(u-w)(x,t)|^2 dx dt,
    \end{array}
\end{equation}
where $c \equiv c(\operatorname{data}_h(\beta), d,\beta)$.
    \end{description}
    Using \cref{II,III,diff1} and convergence property of Steklov average in \cref{differ+}, we get
    \begin{multline}{\label{decay5}}
\fint_{t_0-r^2/16}^{t_0}\fint_{B_{r / 4}(x_0)}|\nabla (u(x,t)-w(x,t))|^2 d x dt +\frac{8}{r^2}\fint_{B_{r/4}(x_0)}\left | (u-w)^2(x,t_0)\right| dx
\leq cr^{2(1-s)}[1+r^{2(\beta-s)}].
\end{multline}
In \cref{decay5} $\beta$ is arbitrary and satisfies $s<\beta<1$, and the constant denoted by $c$ depend on $\beta$, $d$ and $\operatorname{data}_h$(which again depends on $\beta$). Then we take $
\beta:=\frac{1+s}{2}
$
in \cref{decay5} and obtain \cref{decay4}.
\end{proof}
\textbf{Estimates for $\nabla w$:} 
Once \cref{decay4} is established, we now show the local H\"older continuity of $\nabla u$. We recall the Euler-Lagrange equation followed by $w$ as \cref{refsol}, and take the test function as $w_h\phi^2$, for some $\phi\in L^2(t_0-r^2/64,t_0;W_0^{1,2}(B_{r/8}))\cap W^{1,2}(t_0-r^2/64,t_0;L^2(B_{r/8}))$(where $w_h$ is the usual Steklov average), to get 
\begin{equation}{\label{refcacc}}
\begin{array}{rcl}
\fint_{t_0-r^2/64}^{t_0} \fint_{B_{r/ 8}}( w_{h,t}w_h \phi^2  +(\nabla w)_h \cdot \nabla (w_h\phi^2) )d x dt&=&0\\
\implies \frac{1}{2}\fint_{B_{r/ 8}}\fint_{t_0-r^2/64}^{t_0} \frac{\partial w^2_{h}}{\partial t} \phi^2 dt dx +\fint_{t_0-r^2/64}^{t_0} \fint_{B_{r/ 8}} |(\nabla w)_h|^2\phi^2dxdt&=&-2 \fiint_{Q_{r/ 8}} (\nabla w)_hw_h\cdot\phi\nabla\phi dxdt\\&\leq& 2\fiint_{Q_{r/ 8}} |(\nabla w)_h||\phi||w_h||\nabla\phi| dxdt\\&\leq& 2\varepsilon\fint_{t_0-r^2/64}^{t_0} \fint_{B_{r/ 8}} |(\nabla w)_h|^2\phi^2dxdt\\&&+2c(\varepsilon)\fint_{t_0-r^2/64}^{t_0} \fint_{B_{r/ 8}} w_h^2|\nabla \phi|^2dxdt,
\end{array}
\end{equation}
where $c$ is a constant depending only on $\varepsilon$. Taking $\varepsilon=\frac{1}{4}$ in \cref{refcacc}, we get
\begin{equation}{\label{refcaccio}}
    \begin{array}{rcl}
      \fint_{B_{r/ 8}}\fint_{t_0-r^2/64}^{t_0} \frac{\partial w_h^2}{\partial t} \phi^2 dt dx +\fint_{t_0-r^2/64}^{t_0} \fint_{B_{r/ 8}} |(\nabla w)_h|^2\phi^2dxdt &\leq &c\fint_{t_0-r^2/64}^{t_0} \fint_{B_{r/ 8}} |w_h|^2|\nabla \phi|^2dxdt\\
      \implies \fint_{B_{r/ 8}}\fint_{t_0-r^2/64}^{t_0} \frac{\partial (w_h\phi)^2}{\partial t} dt dx+\fint_{t_0-r^2/64}^{t_0} \fint_{B_{r/ 8}} |(\nabla w)_h|^2\phi^2dxdt &\leq &c\fint_{t_0-r^2/64}^{t_0} \fint_{B_{r/ 8}} |w_h|^2|\nabla \phi|^2dxdt\\&&+2\fint_{t_0-r^2/64}^{t_0} \fint_{B_{r/ 8}} |w_h|^2\phi\phi_tdxdt\\ \implies \fint_{B_{r/ 8}} (w_h\phi)^2(x,t_0) dx+\fint_{t_0-r^2/64}^{t_0} \fint_{B_{r/ 8}} |(\nabla w)_h|^2\phi^2dxdt &\leq &c\fint_{t_0-r^2/64}^{t_0} \fint_{B_{r/ 8}} |w_h|^2|\nabla \phi|^2dxdt\\&&+2\fint_{t_0-r^2/64}^{t_0} \fint_{B_{r/ 8}} |w_h|^2\phi\phi_tdxdt\\&&+\fint_{B_{r/ 8}} (w_h\phi)^2(x,t_0-r^2/64) dx.
    \end{array}
\end{equation}
Now choosing $\phi (x,t_0-r^2/64)=0$, $\phi(x,t)\equiv 1$ in $B_{r/16}\times(t_0,t_0-r^2/256)$, $\phi\leq 1$ in $B_{r/8}\times(t_0-r^2/64,t_0)$, $|\partial_t \phi(x,t)|\leq \frac{c}{r^2}$,
and $\left|\nabla \phi\right|\leq \frac{c}{r}$, we get from \cref{refcaccio}, \begin{equation*}
    \fint_{B_{r/ 16}} w_h^2(x,t_0) dx+\fint_{t_0-r^2/256}^{t_0} \fint_{B_{r/ 16}} |(\nabla w)_h|^2dxdt \leq \frac{c}{r^2}\fint_{t_0-r^2/64}^{t_0} \fint_{B_{r/ 8}} |w_h|^2dxdt,
\end{equation*}
where $c$ is a constant. Now using \cref{steklov2}, we get that the right hand side of the above inequality is bounded by $\frac{c}{r^2} \fiint_{Q_{r/ 8}} |w|^2dxdt$. As the left hand side is positive, we can pass to $h\to 0$ and use convergence properties of Steklov averages to get 
\begin{equation}{\label{refercacc}}
    \fint_{B_{r/ 16}} w^2(x,t_0) dx+\fint_{t_0-r^2/256}^{t_0} \fint_{B_{r/ 16}} |\nabla w|^2dxdt \leq \frac{c}{r^2}\fint_{t_0-r^2/64}^{t_0} \fint_{B_{r/ 8}} |w|^2dxdt,
\end{equation}
where $c$ is a constant.
Now we define for $1\leq i\leq n$, $k>0$,
\begin{equation*}
w_{i,k}=\frac{w(x+ke_i,t)-w(x,t)}{k},
\end{equation*}
where $e_i$ represents the standard basis vector of $\mathbb{R}^n$. Clearly all these $w_{i,k}$ are defined for $x\in B_{r/4}^k\equiv\{x\in B_{r/4}:\text{dist}(x,\partial B_{r/4})>k\}$. We choose $k$ small such that $B_{r/4}^k$ is non-empty. Then all these $w_{i,k}$ satisfies the same equation as of $w$ in $B_{r/4}^k\times(t_0-r^2/16,t_0)$. By choosing $k$ sufficiently small we may assume $Q_{r/8}\subset B_{r/4}^k\times(t_0-r^2/16,t_0)$. Then from \cref{refercacc}, using mean value theorem, we get for each $1\leq i \leq n$,
\begin{equation}{\label{gradcacc}}
    \fiint_{Q_{r/ 16}} |\nabla w_{i,k}|^2dxdt \leq \frac{c}{r^2} \fiint_{Q_{r/ 8}} |w_{i,k}|^2dxdt\leq\frac{c}{r^2} \fiint_{Q_{r/ 8}} |\nabla w|^2dxdt,
\end{equation}
where $c$ is a constant depending only on $n$.
Since the right hand side of \cref{gradcacc} is finite, we can pass to $k\to 0$ in left for each $1\leq i\leq n$ to get
\begin{equation}{\label{gradient}}
    \fiint_{Q_{r/ 16}} |\nabla^2 w|^2dxdt \leq \frac{c}{r^2} \fiint_{Q_{r/ 8}} |\nabla w|^2dxdt\leq\frac{c}{r^2} \fiint_{Q_{r/ 4}} |\nabla w|^2dxdt,
\end{equation}
where $c$ is a constant depending only on $n$. Note that \cref{gradient} implies that $\nabla^ 2w$ exists.
Now from \cref{reguw}, we get 
\begin{equation}{\label{reguw2}}
    \fiint_{Q_v} |w_{i,k}(x,t)-(w_{i,k})_{Q_v}|^2dxdt\leq c\left(\frac{v}{r}\right)^2\fiint_{Q_{r/16}} |w_{i,k}(x,t)-(w_{i,k})_{Q_{r/16}}|^2dxdt,
\end{equation}
for all $v\in(0,\frac{r}{16}]$, where $c$ is a constant depending only on $n$. \\Now the right hand side of \cref{reguw2} is bounded by $\left(\frac{v}{r}\right)^2 r^2\fiint_{Q_{r/ 16}} |\nabla w_{i,k}|^2dxdt $ (by \cref{poin}), which is in turn again bounded by $\left(\frac{v}{r}\right)^2 \fiint_{Q_{r/ 4}} |\nabla w|^2dxdt$ (by \cref{gradient,gradcacc}). Also, all the terms in LHS and RHS are positive. So we can pass $k\to 0$, in \cref{reguw2} to get whenever $0<v<\frac{r}{16}$,
\begin{equation}{\label{reguref}}
\begin{array}{rcl}
    \fiint_{Q_v} |\nabla w(x,t)-(\nabla w)_{Q_v}|^2dxdt&\leq &c\left(\frac{v}{r}\right)^2\fiint_{Q_{r/16}} |\nabla w(x,t)-(\nabla w)_{Q_{r/16}}|^2dxdt,
\end{array}
\end{equation}
where $c \equiv c(n) \geq 1$, we have used here the fact that the limit of a non-negative sequence is also non-negative. We will use these estimates of $\nabla w$ to get the regularity of $\nabla u$ as follows.\smallskip\\
\textbf{Regularity of $\nabla u$:} For all $0<v\leq \frac{r}{16}
$, using \cref{av2} and Poincar\'{e} inequality, we estimate
\begin{equation*}
    \begin{array}{cl}
         & \fint_{t_0-v^2}^{t_0}\fint_{B_v}\left|\nabla u-(\nabla u)_{Q_v}\right|^2 d x d t\\\leq &  c\fint_{t_0-v^2}^{t_0}\fint_{B_v}\left|\nabla w-(\nabla w)_{Q_v}\right|^2 d x d t+c\left(\frac{r}{v}\right)^{n+2} \fint_{t_0-r^2/16}^{t_0}\fint_{B_{r/ 4}}|\nabla u-\nabla w|^2d xdt\\\stackrel{\cref{reguref}}{\leq} &c\left(\frac{v}{r}\right)^2\fint_{t_0-r^2/256}^{t_0}\fint_{B_{r/16}}\left|\nabla w-(\nabla w)_{Q_{r/16}}\right|^2 d x d t+c\left(\frac{r}{v}\right)^{n+2} \fint_{t_0-r^2/16}^{t_0}\fint_{B_{r/ 4}}|\nabla u-\nabla w|^2d xdt\\\stackrel{}{\leq} &c \left(\frac{v}{r}\right)^2r^2\fint_{t_0-r^2/256}^{t_0} \fint_{B_{r/16}}\left|\nabla^2 w(x,t)\right|^2 d x dt 
+c\left(\frac{r}{v}\right)^{n+2} \fint_{t_0-r^2/16}^{t_0}\fint_{B_{r/ 4}}|\nabla u-\nabla w|^2d xdt\\\stackrel{\cref{gradient}}{\leq}& c \left(\frac{v}{r}\right)^2\fint_{t_0-r^2/16}^{t_0} \fint_{B_{r/ 4}}|\nabla w(x,t)|^2 d x dt 
+c\left(\frac{r}{v}\right)^{n+2} \fint_{t_0-r^2/16}^{t_0}\fint_{B_{r/ 4}}|\nabla u-\nabla w|^2d xdt\\\stackrel{\cref{decay4}}{\leq}& c \left(\frac{v}{r}\right)^2\Bigg[\fint_{t_0-r^2/16}^{t_0} \fint_{B_{r/ 4}}|\nabla u(x,t)|^2 d x dt+\fint_{t_0-r^2/16}^{t_0} \fint_{B_{r/ 4}}|\nabla u-\nabla w|^2 d x dt\Bigg] 
+c\left(\frac{r}{v}\right)^{n+2}r^{2(1-s)}\\\stackrel{\cref{decay3},\cref{decay4}}{\leq}& c \left(\frac{v}{r}\right)^2[r^{-2\lambda}+r^{2(1-s)}] 
+c\left(\frac{r}{v}\right)^{n+2}r^{2(1-s)}\\\leq & c \left(\frac{v}{r}\right)^2r^{-2\lambda} 
+c\left(\frac{r}{v}\right)^{n+2}r^{2(1-s)} \text {, }
    \end{array}
\end{equation*}
with $c \equiv c(\operatorname{data}_h(\lambda),d, \lambda)$. In the above inequality, since $0<v\leq \frac{r}{16}$ is arbitrary, we take $v=r^{1+(1-s)/(n+2)}/16$ and choose $\lambda=(1-s)/(2n+4)$. Hence we conclude with
\begin{equation}{\label{c1}}
    \fint_{t_0-v^2}^{t_0}\fint_{B_v(x_0)}\left|\nabla u-(\nabla u)_{Q_v}\right|^2 d x d t\leq c v^{2\alpha }, \quad \alpha:=\frac{1-s}{2n+4},
\end{equation}
where $c \equiv c(\operatorname{data}_h, d)$. The above holds whenever $B_v \Subset \Omega$ is a ball with centre in $\Omega_0$, with $v \leq d^{1+(1-s)/(n+2)}/ c(n)$. As per assumption, $\Omega_0 \Subset \Omega_1 \Subset \Omega$ and $(\tau_1,\tau_2)\Subset(t_1,t_2)\Subset(0,T)$ are arbitrary, which proves the local Hölder continuity of $\nabla u$ in $\Omega\times(0,T)$, via the classical Campanato-Meyers' integral characterization along with the estimate for $\|\nabla u\|_{C^{0,\alpha}(\Omega_0\times(\tau_1,\tau_2))}$. 
\nocite{*}
\section*{Acknowledgement} The author would like to thank Dr. Karthik Adimurthi for suggesting the problem and his valuable comments on the results along with detailed discussions. Also, she gratefully acknowledges the financial support provided by IIT Kanpur. This work started during the author's visit to TIFR-CAM, Bengaluru. The author would like to express her gratitude to TIFR-CAM for their invitation and warm hospitality.
\bibliography{main}

\begin{thebibliography}{10}

\bibitem{KA}
Karthik Adimurthi and Sun-Sig Byun.
\newblock Boundary higher integrability for very weak solutions of quasilinear parabolic equations.
\newblock {\em Journal de Math{\'e}matiques Pures et Appliqu{\'e}es}, 121:244--285, 2019.

\bibitem{BL}
Lorenzo Brasco, Erik Lindgren, and Martin Str\"{o}mqvist.
\newblock Continuity of solutions to a nonlinear fractional diffusion equation.
\newblock {\em J. Evol. Equ.}, 21(4):4319--4381, 2021.

\bibitem{frac2}
Stefano Buccheri, Jo{\~a}o~V{\'\i}tor da~Silva, and Lu{\'\i}s~Henrique de~Miranda.
\newblock A system of local/nonlocal p-laplacians: the eigenvalue problem and its asymptotic limit as $p$→$\infty$.
\newblock {\em Asymptotic Analysis}, 128(2):149--181, 2022.

\bibitem{JQ}
Jocemar~Q Chagas, Nicolau~ML Diehl, and Patr{\'\i}cia~L Guidolin.
\newblock Some properties for the {S}teklov averages.
\newblock {\em arXiv:1707.06368}, 2017.

\bibitem{MG}
Cristiana De~Filippis and Giuseppe Mingione.
\newblock Gradient regularity in mixed local and nonlocal problems.
\newblock {\em Mathematische Annalen}, pages 1--68, 2022.

\bibitem{DC}
Agnese Di~Castro, Tuomo Kuusi, and Giampiero Palatucci.
\newblock Local behavior of fractional {$p$}-minimizers.
\newblock {\em Ann. Inst. H. Poincar\'{e} C Anal. Non Lin\'{e}aire}, 33(5):1279--1299, 2016.

\bibitem{frac}
Eleonora Di~Nezza, Giampiero Palatucci, and Enrico Valdinoci.
\newblock Hitchhiker's guide to the fractional {S}obolev spaces.
\newblock {\em Bulletin des sciences math{\'e}matiques}, 136(5):521--573, 2012.

\bibitem{LCE}
Lawrence~C Evans.
\newblock {\em Partial Differential Equations}, volume~19.
\newblock American Mathematical Soc., 2010.

\bibitem{par1}
Yuzhou Fang, Bin Shang, and Chao Zhang.
\newblock Regularity theory for mixed local and nonlocal parabolic {$p$}-{L}aplace equations.
\newblock {\em J. Geom. Anal.}, 32(1):Paper No. 22, 33, 2022.

\bibitem{garain2022regularity}
Prashanta Garain and Juha Kinnunen.
\newblock On the regularity theory for mixed local and nonlocal quasilinear elliptic equations.
\newblock {\em Transactions of the American Mathematical Society}, 375(08):5393--5423, 2022.

\bibitem{ParabolicRegularity}
Prashanta Garain and Juha Kinnunen.
\newblock On the regularity theory for mixed local and nonlocal quasilinear parabolic equations.
\newblock {\em Annali Scuola Normale Superiore - Classe Di Scienze}, page~33, 10 2022.

\bibitem{garain2023weak}
Prashanta Garain and Juha Kinnunen.
\newblock Weak harnack inequality for a mixed local and nonlocal parabolic equation.
\newblock {\em Journal of Differential Equations}, 360:373--406, 2023.

\bibitem{garain2023higher}
Prashanta Garain and Erik Lindgren.
\newblock Higher h{\"o}lder regularity for mixed local and nonlocal degenerate elliptic equations.
\newblock {\em Calculus of Variations and Partial Differential Equations}, 62(2):67, 2023.

\bibitem{giusti2003direct}
Enrico Giusti.
\newblock {\em Direct methods in the calculus of variations}.
\newblock World Scientific, 2003.

\bibitem{GaryM}
Gary~M Lieberman.
\newblock {\em Second order parabolic differential equations}.
\newblock World scientific, 2005.

\bibitem{par2}
Bin Shang and Chao Zhang.
\newblock H\"{o}lder regularity for mixed local and nonlocal {$p$}-{L}aplace parabolic equations.
\newblock {\em Discrete Contin. Dyn. Syst.}, 42(12):5817--5837, 2022.

\bibitem{par3}
Bin Shang and Chao Zhang.
\newblock Harnack inequality for mixed local and nonlocal parabolic {$p$}-{L}aplace equations.
\newblock {\em J. Geom. Anal.}, 33(4):Paper No. 124, 24, 2023.

\bibitem{su2022regularity}
Xifeng Su, Enrico Valdinoci, Yuanhong Wei, and Jiwen Zhang.
\newblock Regularity results for solutions of mixed local and nonlocal elliptic equations.
\newblock {\em Mathematische Zeitschrift}, 302(3):1855--1878, 2022.

\end{thebibliography}
\bibliographystyle{plain}
\end{document}